\documentclass[12pt]{article}
\usepackage{amsmath,bbm, color}
\usepackage{latexsym,mathrsfs,bm}
\usepackage{url}
\usepackage[english]{babel}
\usepackage{paralist}
\usepackage{amsthm,amsfonts,amssymb,amsmath}
\usepackage[latin1]{inputenc}
\usepackage{hyperref}
\usepackage{titlesec}
\usepackage{xcolor}
\definecolor{shadecolor}{RGB}{150,150,150}

\titleformat{\paragraph}
{\normalfont\normalsize\bfseries}{\theparagraph}{1em}{}
\titlespacing*{\paragraph}
{0pt}{3.25ex plus 1ex minus .2ex}{1.5ex plus .2ex}

\newcommand{\eps}{\varepsilon}
\newcommand{\R}{\mathbb{R}}

\newcommand{\I}{\mathcal{I}}
\newcommand{\J}{\mathcal{J}}
\newcommand{\X}{\mathcal{X}}

\newcommand{\Z}{\mathbb{Z}}
\newcommand{\A}{\mathcal{A}}
\newcommand{\es}[1]{\begin{equation}\begin{split}#1\end{split}\end{equation}}
\newcommand{\est}[1]{\begin{equation*}\begin{split}#1\end{split}\end{equation*}}
\newcommand{\as}[1]{\begin{align}#1\end{align}}
\renewcommand{\ast}[1]{\begin{align*}#1\end{align*}}

\newcommand{\M}{\mathcal{M}}

\newcommand{\LL}{\mathcal{L}}

\newcommand{\p}{\mathfrak{p}}
\newcommand{\q}{\mathfrak{q}}

\newcommand{\tn}[1]{\textnormal{#1}}

\renewcommand{\mod}[1]{~\pr{\textnormal{mod}~#1}}
\newtheorem*{theo*}{Theorem}

\newtheorem{theorem}{Theorem}

\newtheorem{lemma}{Lemma}
\newtheorem{corol}{Corollary}
\newtheorem{remark}{Remark}

\newtheorem*{rem*}{Remark}
\newcommand{\pr}[1]{\left( #1\right)}
\newcommand{\prbigg}[1]{\bigg( #1\bigg)}
\newcommand{\prBig}[1]{\Big( #1\Big)}

\newcommand{\pg}[1]{\left\{ #1\right\}}
\newcommand{\pmd}[1]{\left| #1\right|}

\newcommand{\e}[1]{\operatorname{e}\pr{ #1}}

\newcommand{\Nb}{\mathcal{N}}

\def\sumstar{\operatornamewithlimits{\sum\nolimits^*}}

\newcommand{\sumfive}{\operatorname*{\sum\sum\sum\sum\sum}}
\newcommand{\sumtwo}{\operatorname*{\sum\sum}}
\newcommand{\sumthree}{\operatorname*{\sum\sum\sum}}
\newcommand{\sumfour}{\operatorname*{\sum\sum\sum\sum}}
\newcommand{\sumsix}{\operatorname*{\sum\sum\sum\sum\sum\sum}}
\newcommand{\sumseven}{\operatorname*{\sum\sum\sum\sum\sum\sum\sum}}

\newcommand{\comment}[1]{}
\let\originalleft\left
\let\originalright\right
\renewcommand{\left}{\mathopen{}\mathclose\bgroup\originalleft}
\renewcommand{\right}{\aftergroup\egroup\originalright}

\numberwithin{equation}{section}

\makeatletter
\newcommand{\datep}[1]{%
  \let\@oldtitle\@title%
  \gdef\@title{\@oldtitle\footnotetext{\emph{Date:} #1.}}%
}
\newcommand{\subjclass}[2][2010]{%
  \let\@@oldtitle\@title%
  \gdef\@title{\@@oldtitle\footnotetext{#1 \emph{Mathematics subject classification.} #2}}%
}
\newcommand{\keywords}[1]{%
  \let\@@@oldtitle\@title%
  \gdef\@title{\@@@oldtitle\footnotetext{\emph{Key words and phrases.} #1.}}%
}
\makeatother

\newcommand{\addresses}{{
  \bigskip
  \footnotesize

  S.~Bettin, \textsc{Centre de Recherches Math\'ematiques, Universit\'e de Montr\'eal C.P. 6128;
succursale Centre-ville, Montr\'eal (Qu\'ebec); H3C 3J7 Canada
}\par\nopagebreak
  \textit{E-mail address}: \texttt{bettin@crm.umontreal.ca}

  \medskip

  V.~Chandee, \textsc{Department of Mathematics, Burapha University; Chonburi; 20131 Thailand }\par\nopagebreak
  \textit{E-mail address}: \texttt{vorrapan@buu.ac.th}

}}

\begin{document}
\author{Sandro Bettin and Vorrapan Chandee}
\title{Trilinear forms with Kloosterman fractions}


\maketitle

\allowdisplaybreaks
\numberwithin{equation}{section}
\selectlanguage{english}
\begin{abstract}
We give new bounds for $\sum_{{a, m ,n}}\alpha_{m}\beta_n\nu_a\e{\frac{a\overline m}{n}}$ where $\alpha_{m}$, $\beta_n$ and $\nu_a$ are arbitrary coefficients, improving upon a result of Duke, Friedlander and Iwaniec~\cite{DFI97}. We also apply these bounds to problems on representations by determinant equations and on the equidistribution of solutions to linear equations.

\end{abstract}

\section{Introduction}
For $a,b,c$ positive integers, one defines the classical Kloosterman sum as
\est{
S(a,b;c):=\sumstar_{x\mod c}\e{\frac{a\overline x+b x}{c}}
}
where, as usual, $\overline x$ denotes the multiplicative inverse of $x$ modulo the denominator $c$, $\sumstar$ denotes a sum over the reduced residues modulo $c$, and $\e{x}:=e^{2\pi ix}$.

Several important results in number theory have been obtained by using bounds for single Kloosterman sums such as Weil's bound, $S(a,b,c)\ll (a,b, c)^{\frac12}c^{\frac12+\eps}$, or, more recently, for averages of Kloosterman sums, in particular the bounds of Deshouillers and Iwaniec~\cite{DI}.

The results of~\cite{DI} are particularly efficient when considering averages of $S(a,b,c)$ with weights $f(a,b,c)$ that are smooth or have at least some special structure.
For many applications, however, one would like to have non-trivial bounds in the case of arbitrary weights and such bounds would be useful also when the coefficients have arithmetic/geometric nature, but have conductor which is too large to be able to employ the extra information.

In the beautiful paper~\cite{DFI97}, Duke, Friedlander and Iwaniec addressed this problem, obtaining a non-trivial bound for the following ``bilinear form with Kloosterman fractions'':
\est{
\mathcal{B}_a(M,N):=\sumtwo_{\substack{m\in\M ,n\in\Nb ,\\(m,n)=1}}\alpha_{m}\beta_n\e{\frac{a\overline m}{n}},
}
where $\alpha_m, \beta_n$ are arbitrary coefficients supported on $ \M:=[M/2, M]$ and $\Nb:=[ N/2,N]$ respectively, and $ a \neq0$. The main result of~\cite{DFI97} is the bound
\es{\label{ewfc}
\mathcal{B}_a(M,N)&\ll \|\alpha\| \|\beta\| (a+MN)^{\frac38}(M+N)^{\frac{11}{48}+\eps},
}
where $\|\cdot\|$ denotes the $L^2$-norm. Notice that, in the important case $M\approx N$, $a\ll MN$, the bound in~\eqref{ewfc} saves roughly a power of $N^{\frac1{48}}$ over the trivial bound $\mathcal{B}_a(M,N)\ll \|\alpha\| \|\beta\| (MN)^{\frac12}$. In this paper, we refine the arguments of Duke, Friedlander and Iwaniec and improve upon their bound, obtaining a saving of $N^{\frac1{20}}$ when $M\approx N$, $a\ll MN$. More generally, we consider the case when an extra average over $a$ is introduced, as this often appears in applications. We then provide a new bound for sums of the form
\est{
\mathcal{B}(M,N,A):=\sumthree_{\substack{a\in\A, m \in\M,n\in\Nb ,\\(m,n)=1}}\alpha_{m}\beta_n\nu_a\e{\vartheta\frac{a\overline m}{n}},
}
where $\nu_a$ are arbitrary coefficients supported on $\A:=[A/2,A]$ and $\vartheta\neq0$. 

\begin{theorem} \label{thm:boundTri}
Let $\vartheta\neq 0 $. Then 
\es{\label{mrthm:boundTri}
\mathcal{B}(M,N,A)&\ll \|\alpha\| \|\beta\|\|\nu\| \Big(1+\frac{|\vartheta|A}{MN}\Big)^\frac{1}{2}\\
&\hspace{0.1em}\times\hspace{-0.25em}\pr{(AMN)^{\frac7{20}+\eps}(M+N)^{\frac14}+(AMN)^{\frac38+\eps}(AN+AM)^\frac18}.
}
\end{theorem}
Notice that when $|\vartheta|A\ll MN$ and $M \approx N$, the above bound improves the trivial bound $\mathcal{B}(M,N,A)\ll \|\alpha\| \|\beta\|\|\nu\| (AMN)^{\frac12}$ by roughly $\min (A^{\frac 3{20}}N^{\frac1{20}},N^\frac18$).

\begin{remark}\label{ptrmk}
One can perturb slightly the argument of the exponential function and still get the same bound. Indeed, if $f_{a,\vartheta}(x,y)\in\mathcal {C}^1(\R^2)$ is such that 
\es{\label{tckgf}
\frac{\partial}{\partial x}f_{a,\vartheta}(x,y)\ll \frac{X}{x^2y},\qquad \frac{\partial}{\partial y}f_{a,\vartheta}(x,y)\ll \frac{X}{xy^2},\qquad \forall x\in\Nb, \forall y\in\M
} 
for some $X>1$ and any $a,\vartheta$, then when $\theta\neq0$ we have that
\est{
\mathcal{B}_f(M,N,A)&:=\sumthree_{\substack{m\in\M ,n\in\Nb a\in\A,\ (m,n)=1}}\alpha_{m}\beta_n\nu_a\e{\vartheta\frac{a\overline m}{n}+f_{a,\vartheta}(m,n)}
}
satisfies the same bound~\eqref{mrthm:boundTri} of $\mathcal{B}(M,N,A)$ provided that the factor $(1+{|\vartheta|A}/({MN}))^\frac{1}{2}$ is replaced by $(1+(|\vartheta|A+X)/({MN}))^\frac{1}{2}$. 
\end{remark}

Our main motivation for this paper concerns the second moment of the Riemann zeta-function times an arbitrary Dirichlet polynomial,
\est{
I:=\int_{\R}\pmd{\zeta(\tfrac12+it)A(\tfrac12+it)}^2\Phi\pr{t/{T}}\,dt,
}
where $\Phi(x)$ is a test function (supported on $[1,2]$, say) and $A(s):=\sum_{n\leq T^{\theta}}\frac{a_n}{n^{s}}$ with $a_n$ arbitrary coefficients with $a_n\ll n^\eps$. Balasubramanian, Conrey and Heath-Brown \cite{BCH} computed the asymptotic for $I$ when $\theta<\frac12$. In this case only the ``diagonal terms'' contribute. In order to break the $\frac12$ barrier one has to deal with the ``off-diagonal terms'', which quickly leads to the problem of obtaining non-trivial bounds for $\mathcal{B}(M,N,A)$. Equation~\eqref{ewfc} and the stronger Theorem~\ref{thm:boundTri} provide such non-trivial bounds, and so, in a joint work with Radziwi\l\l~\cite{BCR}, we were able to compute the asymptotic for $I$ for $\theta<\frac{17}{33}$. As a comparison, the use of~\eqref{ewfc} would have given the same result on the smaller range $\theta<\frac{48}{95}$. In the same work, we also formulate a conjectural bound for $\mathcal{B}$ which, if true, would allow to extend the range to $\theta<1$ and thus imply the Lindel\"of hypothesis.

The flexibility of Theorem~\ref{thm:boundTri} makes it feasible to be applied to a wide class of problems in number theory. Moreover, the strength of the new bound is now competitive even in some cases when one knows and could potentially employ some information on the coefficients. 

We also give two easy applications of Theorem~\ref{thm:boundTri}. One could use the new bound also to improve some results proved using~\eqref{ewfc} (e.g.~\cite{DFI12} and~\cite{BS}) or, possibly, to sharpen sub-convexity bounds for automorphic $L$-functions (see~\cite{DFI02}) or to handle sums such as those considered by Fouvry~\cite{Fou} and Bombieri, Friedlander and Iwaniec~\cite{BFI}.

The first corollary deals with representations by determinant equations and improves the main result of~\cite{DFI95}, whereas the second concerns with the equidistribution of solutions to linear equations and improves upon a theorem of Shparlinski~\cite{Shp}.
\begin{corol}\label{c1}
Let $\Delta\neq0$ and let
\est{
\mathcal{T}(M_1,M_2,N_1,N_2):=\sumfour_{\substack{m_1\in\M_1,m_2\in\M_2, n_1\in\Nb_1,n_2\in\Nb_2,\\ m_1n_2-m_2n_1=\Delta}}f(m_1)g(m_2)\alpha_{n_1}\beta_{n_2},
}
where $f(m_1)$, $g(m_2)$, $\alpha_{n_1}$ and $\beta_{n_2}$ are supported on $\M_{1}:=[M_1/2,M_1]$, $\M_{2}:=[M_2/2,M_2]$, $\Nb_{1}:=[N_1/2,N_1]$ and $\Nb_{2}:=[N_2/2,N_2]$ respectively. Moreover, assume
$f^{(j)}\ll \eta^{j}M_1^{-j}, g^{(j)}\ll \eta^{j}M_2^{-j}$, for all $j\geq0$ and some $\eta>1$. Then
\es{\label{repd}
\mathcal{T}(M_1,M_2,N_1,N_2)&=\sum_{\substack{n_1\in\Nb_1,n_2\in\Nb_2,\\(n_1,n_2)|\Delta}}\frac{(n_1,n_2)}{n_1n_2}\alpha_{n_1}\beta_{n_2}\int_\R f\pr{\frac{x+\Delta}{n_2}}g\pr{\frac x{n_1}}\,dx\\
&\quad+O\bigg((\eta R)^{\frac{3}{2}}\|\alpha\|\|\beta\| (N_1N_2)^{\frac7{20}}(N_1+N_2)^{\frac14+\eps} (M_1M_2)^\eps\bigg),
}
where $R:=\frac{M_1N_2}{M_2N_1}+\frac{M_2N_1}{M_1N_2}$.
\end{corol}
For comparison, Duke, Friedlander and Iwaniec~\cite{DFI95} obtained the same result with the error term
\est{
O\Big((\eta R)^{\frac{19}8}\|\alpha\|\|\beta\|(N_1N_2)^{\frac 38}(N_1+N_2)^{\frac{11}{48}+\eps}(M_1M_2)^\eps\Big).
}
We remark that one can use Theorem~\ref{thm:boundTri} to obtain stronger results when averaging over $\Delta$ (cf.~\cite{BCR}).

\begin{corol}
For any positive coprime integers $m,n$, let $\rho_{m,n}:=\frac{a_0}{m}$ where $(a_0,b_0)$ is the smallest positive solution to $am-bn=1$. For any  set of integers $\X_N\subseteq [0,N]$, let $\mathcal R_{\X_N}:=\pg{\rho_{m,n}\mid  m,n\in\X_N}$. Then, if $\X_N$ has cardinality $|\X_N|\gg N^{1-\frac{1}{20}}$ for some $\eps>0$, then the set $\mathcal R_{\X_N}$ is equidistributed on the interval $[0,1]$ as $N\rightarrow\infty$.
\end{corol}
Shparlinski obtained the same result with $\frac{1}{20}$ replaced by $\frac1{48}$. We skip the proof of this corollary as it can be obtained in a straightforward manner by proceeding as in~\cite{Shp}, using~\eqref{mrthm:boundTri} instead of~\eqref{ewfc}. 

As observed in~\cite{DFI97}, a variation of the arguments used to bound $\mathcal{B}(M,N,A)$ can be used to treat the twisted sum
\est{
\mathscr{A}(M,N,A):=\sumthree_{\substack{a\in\A, m\in\M ,n\in\Nb ,\\(m,n)=(2,mn)=1}}\alpha_{m}\beta_n\nu_a\pr{\frac mn}\e{\vartheta\frac{a\overline m}{n}},
}
where $\pr{\frac{\cdot}{\cdot}}$ is the Jacobi symbol. We thus conclude the introduction with the analogue of Theorem~\ref{thm:boundTri} for $\mathscr{A}(M,N,A)$.

\begin{theorem}\label{thtw}
Let $\vartheta\neq0$. Then
\es{\label{acaavcf}
 \mathscr{A}(M,N,A)&\ll \|\alpha\|\|\beta\|\|\nu\|  \Big(1+ \frac{|\vartheta|A}{NM}\Big)^{\frac12} \\
 &\quad\times\big( (MN)^{\frac3{10}}(AM+AN)^{\frac{7}{20}+\eps}+A^\frac12(N+M)^{\frac78+\eps}\big).
}
\end{theorem}

\textbf{Acknowledgements}. 
We would like to thank S. Drappeau, A. Harper, D. Koukoulopoulos, X. Li and M. Radziwi\l\l{} for useful discussions and M.B. Milinovich and N. Ng for bringing the paper~\cite{DFI97} to our attention.

\section{Outline of the proof of Theorem~\ref{thm:boundTri}}\label{out}
Our proof has roughly the same structure of Duke, Friedlander and Iwaniec's proof of~\eqref{ewfc} and follows their clever application of the amplification method. However we introduce several refinements in their arguments, among which is particularly important the fact that we keep a longer diagonal when using the Cauchy-Schwartz inequality (a possibility mentioned in~\cite{DFI97}). This change, together with the extra average over $a$, introduces new subtle complications and requires a rather careful analysis. 

We now give an outline of the proof of Theorem~\ref{thm:boundTri}. The proof of Theorem~\ref{thtw} is very similar and the required changes are described in Section~\ref{wec}.

First, we notice that we can assume that $\beta_n$ is supported on integers coprime to $\vartheta$, as can be seen by pulling out the common factor between $n$ and $\vartheta$ and applying the Cauchy-Schwarz inequality.
Then, following~\cite{DFI97}, we 
apply the Cauchy-Schwarz inequality to the sum over $m$ and obtain that
\es{\label{bfc}
\mathcal{B}(M,N,A) \ll \|\alpha\|\mathcal{C}_1(M,N,A;\beta,\nu)^{\frac12},
}
where for any positive integer $b$,
\est{
\mathcal{C}_b(M,N,A;\beta,\nu):=\sum_{\substack{m\in\M,\\(m,b)=1}}\bigg|\sum_{\substack{a\in\A}}\sum_{\substack{n\in\Nb,\\(m,n)=1}}\beta_n\nu_a\e{\vartheta\frac{ a\overline m}{bn}}\bigg|^2.
}
At this point, we assume that $\beta_n$ is supported on square-free integers coprime to $b$ and that $(\vartheta,b)=1$. We will first give a bound for $\mathcal{C}_b$ in this case and in Section~\ref{rsq} we will use the freedom given by the parameter $b$ to obtain a bound for $\mathcal{C}_1$ valid in the general case.

We introduce an amplifier and consider the sum $\mathcal{D}_b(M,N,A,L;\beta,\nu)$ 
defined by
\est{
\mathcal{D}_b:=\!\!\!\sum_{\substack{m\in \M,\\(m,b)=1}}\frac1{\varphi(m)}\sum_{\chi\mod m}\hspace{-0.4em}\big|\!\sum_{\substack{\ell\in \LL,\\ (\ell,\vartheta b)=1}}\hspace{-0.4em}\chi(\ell)\big|^2 \bigg|\sumtwo_{\substack{n\in \Nb,a\in\A,\\(m,n)=1}}\chi(n)\beta_n\nu_a\e{\vartheta\frac{a\overline m}{bn}}\bigg|^2,
}
where $\LL:=\{\ell \ \textrm{is prime} \ | \ L < \ell < 2L\}$ and $L$ is a parameter to be chosen at the end of the argument. If $L> 2\log (b\vartheta M)$, then 
\est{
\sum_{\ell\in\LL,\ (\ell,\vartheta b)=1}\chi_0(\ell)\gg \frac{L}{\log L},
}
where $\chi_0$ is the principal character modulo $m$. Thus, 
\es{\label{amp}
\mathcal{C}_b&\ll \frac{M\log^2 L}{L^2}\hspace{-0.2em}\sum_{\substack{m\in\M,\\(m,b)=1}}\frac1{\varphi(m)}\bigg|\sum_{\substack{(\ell,\vartheta b)=1,\\\ell\in \LL}}\chi_0(\ell)\bigg|^2 \bigg|\hspace{-0.4em}\sum_{\substack{n\in \Nb,\\(m,n)=1}}\sum_{a\in \A}\beta_n\nu_a\e{\vartheta\frac{a\overline m}{bn}}\bigg|^2\\
&\ll  {ML^{-2+\eps}}\mathcal{D}_b(M,N,A,L; \beta,\nu),
}
provided that $L> 2\log (b\vartheta M)$. Thus, we have reduced the problem of bounding $\mathcal{C}_b$ to that of bounding the more flexible $\mathcal{D}_b$.  Squaring out and exploiting the orthogonality relation of character sums, we obtain
\as{ \label{def:DbEb}
\mathcal{D}_b(M,N,A;\beta,\nu)&=\sumseven_{\substack{m\in\M,n_1,n_2\in\Nb,a_1,a_2\in\A,\ell_1,\ell_2\in\LL,\\
(mb\vartheta,\ell_1\ell_2n_1n_2)=(m,b)=1,\\\ell_1n_1\equiv \ell_2n_2\mod m}}\beta_{n_1}\nu_{a_1}\overline {\beta_{n_2}\nu_{a_2}}\e{\vartheta\frac{a_1\overline m}{bn_1}-\vartheta\frac{a_2\overline m}{bn_2}} \notag \\
&= \mathscr{D}_b(M,N, A,L; \beta, \nu) + \mathscr{O}_b(M,N, A,L; \beta, \nu),
}
where $\mathscr{D}_b$ is the contribution to $\mathcal{D}_b$ from the ``diagonal terms'' $\ell_1n_1 = \ell_2n_2$, and $\mathscr{O}_b$ is the sum restricted to the ``off-diagonal'' terms $\ell_1n_1 \neq \ell_2n_2.$ 

We bound $\mathscr{D}_b$ in Section~\ref{sec:diagonal_ln_equald} by using Weil's bound (and thus the name ``diagonal terms'' is perhaps misleading in this case), treating it differently from~\cite{DFI97} where $\mathscr{D}_b$ is bounded trivially.

The treatment of $\mathscr{O}_b$ is performed in Section~\ref{sec:offdiagonal:initial} and the argument proceeds roughly this way: 
\begin{itemize}
\item In Section~\ref{gfdsa}, we switch to the complementary divisor $d$ of the congruence relation $\ell_1n_1\equiv \ell_2n_2\mod m$, eliminating the variable $m$. This also requires that we first split the sum over $m$ into certain congruence classes.
\item In Section~\ref{csi}, we apply the Cauchy-Schwarz inequality to the sums over $n_1,n_2,a_2$ but not to the sums over $d,a_1,\ell_1,\ell_2$. As a comparison, in~\cite{DFI97} the Cauchy-Schwarz inequality is applied to all the sums except those over $\ell_1$ and $\ell_2$.
\item In Section~\ref{mccs}, we apply the elementary reciprocity law~\eqref{cond:3congruence}, which roughly allows one to change $\frac{\overline \alpha}{\beta}$ into  $-\frac{\overline\beta }{\alpha}$ modulo 1, and so we arrive to an expression involving a sum of the form $\sum_{n_2}\operatorname{e}(\Delta\frac{ \overline {n_2}}{\ell_1\ell_1'n_1})$. We then use Weil's bound when $\Delta\neq0$ and a trivial estimation when $\Delta=0$.
\end{itemize}

When following the above steps, several complications arise when some of the variables are not pairwise coprime. Typically, we first deal with the case when the variables are coprime and then we remove these assumptions by using the bounds proved in the coprime case.

In Section~\ref{optl} we combine the bounds obtained for the diagonal and off-diagonal terms, and we optimize the parameter $L$. Finally, in Section~\ref{rsq} we remove the square-free condition on $\beta_n$, and in Section~\ref{rsq3} we deduce Theorem~\ref{thm:boundTri}.

As mentioned above, we conclude the paper by explaining, in Section~\ref{wec}, the modifications needed in the above arguments to prove Theorem~\ref{thtw} and, in Section~\ref{pc1}, we give a proof to Corollary~\ref{c1}.

\begin{rem*}
Throughout the paper, we use the common convention in analytic number theory that $\eps$ denotes an arbitrarily small positive quantity that may vary from line to line.
\end{rem*}

\section{The diagonal terms} \label{sec:diagonal_ln_equald}
In this section we bound the diagonal terms $\mathscr{D}_b:=\mathscr{D}_b(M,N,A,L;\beta,\nu)$. As mentioned in Section~\ref{out}, we assume that $(\vartheta,b)=1$ and that $\beta_n$ is supported on square-free integers which are coprime to $b$. For convenience of notation, we will also assume that $A,b,\vartheta,N \ll M^{C}$ for some constant $C>0$. We will remove these assumptions at the end of the argument.

By symmetry and the inequality $2|ab|\leq a^2+b^2$ , we see that the diagonal terms are bounded by
\as{
\mathscr{D}_b&:= \sumseven_{\substack{m\in\M,\ell_1, \ell_2\in\LL, n_1, n_2\in\Nb,a_1,a_2\in\A \\
(mb\vartheta,\ell_1\ell_2n_1n_2)=(m,b)=1,\ \ell_1n_1= \ell_2n_2}}\nu_{a_1}\overline{\nu_{a_2}}\beta_{n_1}\overline \beta_{n_2}\e{\vartheta \frac{a_1\overline m}{bn_1}-\vartheta \frac{a_2\overline m}{bn_2}}\notag\\
&\leq \sumsix_{\substack{\ell_1, \ell_2\in\LL, n_1, n_2\in\Nb,a_1,a_2\in\A \\ (b\vartheta,\ell_1\ell_2n_1n_2)=1,\ \ell_1n_1= \ell_2n_2}} (|\beta_{n_1}\nu_{a_1}|^2+|\beta_{n_2}\nu_{a_2}|^2 ) \times \notag
\\[-0.5em]
& \hskip 13em \times \bigg|\sum_{\substack{m\in\M,\\(m,b\ell_1\ell_2n_1n_2) = 1}}\e{\vartheta\frac{a_1\ell_1\overline m}{b\ell_1n_1}-\vartheta\frac{a_2\ell_2\overline m}{b\ell_2n_2}} \bigg|  \notag\\
&\ll \sumsix_{\substack{\ell_1, \ell_2\in\LL, n_1, n_2\in\Nb,a_1,a_2\in\A \\ (b\vartheta,\ell_1\ell_2n_1n_2)=1,\ \ell_1n_1= \ell_2n_2}} |\beta_{n_1}\nu_{a_1}|^2\bigg|\sum_{\substack{m\in\M,\\(m,b\ell_1\ell_2n_1n_2) = 1}}\e{\vartheta\frac{(a_1\ell_1 - a_2\ell_2)\overline m}{b\ell_1n_1}} \bigg|.\notag 
 }
For the terms satisfying $a_1\ell_1 \neq a_2\ell_2$ we use the version of Weil's bound given in Lemma \ref{ks}, in the appendix. We obtain
\est{
\sum_{\substack{m\in\M,\\(m,b\ell_1\ell_2n_1n_2) = 1}}\e{\vartheta\frac{(a_1\ell_1 - a_2\ell_2)\overline m}{b\ell_1n_1}}\ll (bLN)^{\frac 12+\eps} + (a_1\ell_1-a_2\ell_2,bn_1\ell_1) \frac{M^{1+\eps}}{bLN},
}
since $(\vartheta,b\ell_1n_2)=1$.
It follows that the contribution to $\mathscr{D}_b$ coming from these terms is bounded by  
\est{
&\ll \sumfive_{\substack{\ell_1, \ell_2\in\LL, n_1\in\Nb,a_1,a_2\in\A, \ \  \ell_2|\ell_1n_1}} |\beta_{n_1}\nu_{a_1}|^2(bLN)^{\frac 12+\eps}+{}\\
&\quad+  \sumthree_{\substack{ n_1\in\Nb,a_1,a_2\in\A }}|\beta_{n_1}\nu_{a_1}|^2\frac{M^{1+\eps}}{bLN}  \sumtwo_{\substack{\ell_1, \ell_2 \\ a_1\ell_1\neq a_2\ell_2,\, \ell_2|\ell_1n_1} }(a_1\ell_1-a_2\ell_2,bn_1)(a_2\ell_2, \ell_1) \\
&\ll \|\beta\|^2\|\nu\|^2 ALM^{\eps}\Big((bLN)^{\frac 12}+\frac{M}{bN}\Big),
}
since
\es{\label{adg}
\sumthree_{\substack{\ell_1, \ell_2, a_2,\\   a_1\ell_1\neq a_2\ell_2,\  \ell_2|\ell_1n_1}}\hspace{-0.4em}(a_1\ell_1-a_2\ell_2,bn) &\ll  \sum_{ \substack{|c|\leq 4AL,\\c\neq0}}\hspace{-0.2em}(c,bn) \sumthree_{\substack{\ell_1, \ell_2, a_2,\\   a_1\ell_1- a_2\ell_2=c,\\  \ell_2|\ell_1n_1}} (a_2\ell_2, \ell_1)\\
&\ll  AL^2M^\eps.
}
The contribution to $\mathscr{D}_b$ coming from the terms with $a_1\ell_1=a_2\ell_2$ is trivially $O(\|\beta\|^2 \|\nu\|^2 L^{1 + \eps} M)$, and thus
\es{ \label{eqn:diagonal_ellnequal}
\mathscr{D}_b(M,N,A,L;\beta,\nu) &\ll \|\beta\|^2\|\nu\|^2 L\Big(A(bLN)^{\frac 12}+\frac{AM}{bN}+M\Big)M^{\eps}.
}

\section{The off-diagonal terms} \label{sec:offdiagonal:initial}

In this section we bound the off-diagonal terms $\mathscr{O}_b(M,N,A,L;\beta,\nu)$. Again, we have the same assumptions for $\beta_n,\vartheta, b, A, N$ as in Section \ref{sec:diagonal_ln_equald}.

We start by dividing $ \mathscr{O}_{b} $ according to whether $(\ell_1,\ell_2)=1$ or not:  
\es{\label{gyui}
 \mathscr{O}_{b}(M,N,A,L; \beta,\nu) & =\mathscr{E}_{b,1}(M,N,A,L; \beta,\nu) +\mathscr{E}_{b,1}^*(M,N,A,L; \beta,\nu),
}
where, for any $\eta$ satisfying $(\eta,b)=1$ and $(\eta,\ell)=1$ for all $\ell\in\LL$, we define
\es{\label{fafe}
\mathscr{E}_{b,\eta} :=\sumseven_{\substack{\ell_1,\ell_2\in\LL,m\in\M,n_1 , n_2\in \Nb,a_1,a_2\in\A \\(mb\vartheta,\ell_1\ell_2n_1n_2)=(m,b)=(\ell_1,\ell_2)=1,\ \eta|m,\\\ell_1n_1\equiv \ell_2n_2\mod m, \  \ell_1 n_1\neq \ell_2n_2}}\hspace{-0.3em}\beta_{n_1}\nu_{a_1}\overline {\beta_{n_2}\nu_{a_2}}\e{\vartheta\frac{a_1\overline m}{bn_1}-\vartheta\frac{a_2\overline m}{bn_2}},
}
and $\mathscr{E}^*_{b,\eta}(M,N,A,L; \beta,\nu)$ is the same sum with $(\ell_1,\ell_2)=1$ replaced by $(\ell_1,\ell_2)>1$. We notice that 
\es{\label{tqv}
\mathscr{E}^*_{b,1}
 =\sum_{1<\ell\in\LL}\pr{\mathscr{E}_{b,1}(M,N,A,\{1\}; \beta,\nu)-\mathscr{E}_{b,\ell}(M,N,A,\{1\}; \beta,\nu)},
}
where we extended the definition of $\mathscr{E}_{b,\eta}$ to the case where $\LL=\{1\}$. Thus it suffices to bound $\mathscr{E}_{b,\eta}$.

We introduce some notation:
\es{\label{vae1}
n_1'&:=\frac{n_1}{(n_1,\ell_1\ell_2)}=\frac{n_1}{\p_1\p_2},\quad  \p_2=(\ell_2,n_1),\quad \p_1=(\ell_1,n_1),\\
n_2'&:=\frac{n_1}{(n_2,\ell_1\ell_2)}=\frac{n_2}{\q_1\q_2},\quad \q_1=(\ell_1,n_2),\quad \q_2=(\ell_2,n_2),\\
}
and notice that, for square-free integers $n_1,n_2$, this automatically gives $(\ell_1\ell_2,n_1'n_2')=1$.
We also divide $\mathscr{E}_{b,\eta}(M,N,A,L; \beta,\nu)$ further into
\es{\label{dft}
\mathscr{E}_{b,\eta}=\mathscr{S}_{b,\eta}+\mathscr{S}^*_{b,\eta},
}
where $\mathscr{S}_{b,\eta}$ and $\mathscr{S}^*_{b,\eta}$ are obtained by restricting the sums to $(n_1',n_2')=1$ and $(n_1',n_2')>1$ respectively. 

\subsection{The terms with $(n_1',n_2')=1$}

In this section we consider the sum
\est{
&\mathscr{S}_{b,\eta} :=\sumseven_{\substack{\ell_1, \ell_2\in\LL, n_1 , n_2\in \Nb,m\in\M ,a_1,a_2\in\A\\(mb\vartheta,\ell_1\ell_2n_1n_2)=(m,b)=(\ell_1,\ell_2)=(n_1',n_2')=1\\\ell_1n_1\equiv \ell_2n_2\mod m,\  (\ell_1\ell_2,n_1'n_2')=1,\\ \ell_1 n_1\neq \ell_2n_2,\ \eta|m}}\hspace{-0.3em} \beta_{n_1}\nu_{a_1}\overline {\beta_{n_2}\nu_{a_2}}\e{\vartheta\frac{a_1\overline m}{bn_1}-\vartheta\frac{a_2\overline m}{bn_2}}.
}

\subsubsection{Introducing the complementary divisor}\label{gfdsa}

We wish to switch to the complementary divisor of the congruence condition $\ell_1n_1\equiv \ell_2n_2\mod m$ (with $\ell_1n_1 \neq \ell_2 n_2$), so we write this as
\est{
md_0 = 
\ell_1\p_1\p_2n_1' - \ell_2\q_1\q_2n_2',\qquad d_0\in\Z_{\neq0}
}
with $\p_i,\q_i$ as in~\eqref{vae1}. We simplify the common factors and rewrite this equality as
\es{\label{defreduced*2}
md = \tilde \ell_1 \p_1 n_1' -\tilde \ell_2  \q_2n_2'
}
for $d:={d_0}/{\q_1\p_2}$ and
\es{\label{vae2}
 \tilde \ell_1:=\frac{\ell_1}{\q_1}, \quad \tilde \ell_2:=\frac{\ell_2}{\p_2 }.
}
We notice that the condition $(m,b\ell_1\ell_2n_1n_2)=1$ can be factored into $(m,b\q_1\p_2)\linebreak[0]=1$ and $(m,\tilde \ell_1   \tilde \ell_2  \p_1 \q_2n_1'n_2')=1$. Moreover, for $(\ell_1n_1',\ell_2n_2')=1$,
the conditions $(m,\tilde \ell_1   \tilde \ell_2  \p_1 \q_2n_1'n_2' )=1$, $\eta|m$ and~\eqref{defreduced*2} can be expressed in the equivalent form
\ast{
&(d,\tilde \ell_1   \tilde \ell_2  \p_1 \q_2n_1'n_2')=1,&&
\tilde \ell_1 \p_1 n_1' \equiv \tilde \ell_2  \q_2n_2'\mod {|d|\eta},\\
&\q_1\p_2|d|\leq D:=3\frac{NL}M,&& n_2'\in \I,
}
for a certain interval $\I:=\I(\ell_1, \ell_2 , \p_1,\p_2,\q_1,\q_2,d,n_2')$. Thus, to eliminate the variable $m$, it remains to express the condition $(m,b\q_1\p_2)=1$ and the argument of the exponential in terms of the remaining variables.

We do this by dividing the sum over $m$ according to the residue classes $m\equiv c\mod{b\q_1\p_2}$ for $c\in(\Z/b\q_1\p_2\Z)^*$. Thus, using also~\eqref{defreduced*2}, we obtain that $m$ satisfies the following congruence conditions
\est{
m\equiv \overline d \tilde \ell_1 \p_1 n_1' \mod{  \q_2n_2'},\hspace{0.5em}
m\equiv  - \overline d \tilde \ell_2  \q_2n_2'\mod { \p_1 n_1'},\hspace{0.5em} m\equiv c \mod {b\q_1\p_2}
}
and the argument of the exponential function 
\est{
\vartheta\pr{\frac{a_1\overline m}{bn_1}-\frac{a_2\overline m}{bn_2}}=\vartheta\pr{\frac{a_1\overline m}{b\p_1\p_2n_1'}-\frac{a_2\overline m}{b\q_1\q_2n_2'}}
}
is congruent modulo 1 to
\as{
&\vartheta\prbigg{-d\frac{a_1\overline {b\p_2 \tilde \ell_2  \q_2n_2'}}{\p_1n_1'}-d\frac{a_2\overline {\tilde \ell_1 \p_1 n_1'  b\q_1}}{\q_2n_2'}+\frac{a_1\overline {c\p_1n_1'}}{b\p_2}-\frac{a_2\overline {c\q_2n_2'}}{b\q_1}}\notag\\
&\hspace{2em}\equiv\vartheta\prbigg{\overline c \prbigg{\frac{a_1\overline {\p_1n_1'}}{b\p_2}-\frac{a_2\overline {\q_2n_2'}}{b\q_1}}-d\prbigg{\frac{a_1\overline {b\p_2 \tilde \ell_2  \q_2n_2'}}{\p_1n_1'}+\frac{a_2\overline {\tilde \ell_1 \p_1 n_1'  b\q_1}}{\q_2n_2'}}},\label{vae}
}
since for $(\beta,\gamma)=(\alpha,\beta\gamma)=1$, we have $\frac{ \overline{\alpha}}{\beta \gamma}\equiv \frac{\overline{ \alpha\beta}}{\gamma}+\frac{ \overline {\alpha\gamma}}{\beta}\mod1.$ Finally, the conditions $m\equiv c\mod{b\q_1\p_2} $ and $\tilde \ell_1 \p_1 n_1' \equiv\tilde \ell_2  \q_2n_2' \mod {|d|\eta}$ can be combined into the equivalent 
\est{
\tilde \ell_1 \p_1 n_1'\equiv \tilde \ell_2  \q_2n_2'-c\eta d\mod {b\q_1\p_2|d|\eta}.
}
Thus,
\est{
\mathscr{S}_{b,\eta} &=\sumseven_{\substack{\ell_1, \ell_2\in\LL, n_1 , n_2\in \Nb, a_1,a_2\in \A, d\leq {D}/{\q_1\p_2}\\(b\vartheta,\ell_1\ell_2n_1n_2)=(\ell_1,\ell_2)=(n_1',n_2')=1,\\
(d,\tilde \ell_1   \tilde \ell_2  \p_1 \q_2n_1'n_2')=(n_1'n_2',\ell_1\ell_2)=1,\  n_2'\in \I,\\ \tilde \ell_1 \p_1 n_1' \neq \tilde \ell_2  \q_2n_2'}}
\sumstar_{\substack{c\mod{b\q_1\p_2},\\\tilde \ell_1 \p_1 n_1' \equiv \tilde \ell_2  \q_2n_2'-cd\eta\\\hspace{1.5em}\mod {b\q_1\p_2|d|\eta}  }}\beta_{n_1}\nu_{a_1}\overline {\beta_{n_2}\nu_{a_2}}\e{\cdots},\\
}
where the argument of the exponential is given by~\eqref{vae} and $n_i',\p_i,\q_i,\tilde \ell_i$ are defined in~\eqref{vae1} and~\eqref{vae2}. Now, we treat $\p_1,\p_2,\q_1,\q_2,n_1',n_2'$ as variables and, after switching the order of summation, we have
\est{
\mathscr{S}_{b,\eta} &=
 \sumfour_{\substack{\p_1,\p_2,\q_1,\q_2\in\LL\cup\{1\},\\(b\vartheta\p_1\q_1,\p_2\q_2)=1,\\(b\vartheta,\p_1q_1)=1
 }}
 \sumtwo_{\substack{ \p_1\p_2n_1',\q_1\q_2n_2'\in \Nb,\\(b\vartheta,n_1'n_2')=1,\\(n_1',n_2')=1}}
 \sumtwo_{\substack{\ell_1,\ell_2\in\LL,\\ \p_1,\q_1|\ell_1,\, \p_2,\q_2|\ell_2,\\(bn_1'n_2'\vartheta,\ell_1\ell_2)=(\ell_1,\ell_2)=1,\\ \tilde \ell_1 \p_1 n_1' \neq \tilde \ell_2  \q_2n_2'}}
 \sum_{\substack{0\neq |d|\leq D/\q_1\p_2,\\(d,\tilde \ell_1   \tilde \ell_2  \p_1 \q_2n_1'n_2')=1\\ \ n_2'\in \I}}\\
& \sumstar_{\substack{c\mod{b\q_1\p_2},\\\tilde \ell_1 \p_1 n_1' \equiv \tilde \ell_2  \q_2n_2'-cd\eta\mod {b\q_1\p_2|d|\eta}  }}
\sumtwo_{a_1, a_2\in \A}\beta_{\p_1\p_2n_1'}\nu_{a_1}\overline {\beta_{\q_1\q_2n_2'}\nu_{a_2}}\e{\cdots},\\
}
with the argument of the exponential still given by~\eqref{vae}.

\subsubsection{Applying the Cauchy-Schwarz inequality}\label{csi}
Next, we apply the Cauchy-Schwarz inequality with respect to the sums over $\p_1,\p_2,\q_1,\q_2,n_1',n_2',c,a_2$. After squaring out, we get
\es{\label{aaebfr}
&\mathscr{S}_{b,\eta} ^2 \ll M^\eps\|\beta\|^4\|\nu\|^2 \sumfour_{\substack{\p_1,\p_2,\q_1,\q_2\in\LL\cup\{1\},\\\p_1\neq\q_1 \Rightarrow 1\in\{\p_1,\q_1\},\\ \p_2\neq\q_2 \Rightarrow 1\in\{\p_2,\q_2\}}}b\q_1\p_2\mathscr{T}_{b,\eta},
}
where $\mathscr{T}_{b,\eta}$ is defined as 
\es{\label{aae}
\mathscr{T}_{b,\eta}&:=
 \sumtwo_{\substack{\p_1\p_2n_1', \q_1\q_2n_2'\in \Nb\\(b,n_1'n_2')=(n_1',\vartheta n_2')=1,\\ \mu^2(n_1')=1}}
 \sumfour_{\substack{\ell_1,\ell_2,\ell_1',\ell_2'\in\LL,\\ \p_1,\q_1|(\ell_1,\ell_1'),\ \p_2,\q_2|(\ell_2,\ell_2')\\(b\vartheta n_1'n_2',\ell_1\ell_2\ell_1'\ell_2')=(\ell_1,\ell_2)=(\ell_1',\ell_2')=1,\\ \tilde \ell_1 \p_1 n_1' \neq \tilde \ell_2  \q_2n_2',\ \tilde \ell_1' \p_1 n_1' \neq \tilde \ell_2'  \q_2n_2'}}
 \sumtwo_{\substack{0\neq |d|,|d'|\leq D/\q_1\p_2,\\(d,\tilde \ell_1   \tilde \ell_2  \p_1 \q_2n_1'n_2')=1\\ (d',\tilde \ell_1'   \tilde \ell_2'  \p_1 \q_2n_1'n_2')=1\\ \ n_2'\in \I\cap\I'}}\\
 &\quad  \sumstar_{\substack{c\mod{b\q_1\p_2},\\\tilde \ell_1 \p_1 n_1' \equiv \tilde \ell_2  \q_2n_2'-cd\eta\mod {b\q_1\p_2|d|\eta},\\ \tilde \ell_1' \p_1 n_1' \equiv \tilde \ell_2'  \q_2n_2'-cd'\eta\mod {b\q_1\p_2|d'|\eta}  }}
\sumthree_{a_1,a_1', a_2\in \A}\nu_{a_1}\overline {\nu_{a_1'}}\e{\cdots},\\
}
with $\I':=\I(\ell_1', \ell_2' , \p_1,\p_2,\q_1,\q_2,d',n_2')$, $ \tilde \ell_1':=\ell_1'/{\q_1}, \tilde \ell_2':=\ell_2'/{\p_2 }$ and where we introduced the condition $\mu^2(n_1')=1$, which was implicit in the previous formulae, and we dropped the condition $(n_2',\vartheta)=1$. 
The argument of the exponential is now
\es{\label{vha}
\vartheta\prbigg{\overline c (a_1-a_1')\frac{\overline {\p_1n_1'}}{b\p_2}-\frac{(da_1{\tilde \ell_2'}-d'a_1'{\tilde \ell_2})\overline {\tilde \ell_2\tilde \ell_2'b\p_2   \q_2n_2'}}{\p_1n_1'}-\frac{a_2(d \tilde{ \ell_1'}-d' \tilde{ \ell_1})\overline{ \tilde \ell_1\tilde \ell_1'\p_1 n_1'  b\q_1}}{\q_2n_2'}}.
}
We divide the right hand side of~\eqref{aae} into two parts:
\es{\label{eesbf}
&\mathscr{T}_{b,\eta} = \mathscr{U}_{b,\eta}+\mathscr{U}^*_{b,\eta},
}
where $\mathscr{U}_{b,\eta}$ is the contribution of the terms with $(\ell_1\ell_1',\ell_2\ell_2')=1$. We will bound $\mathscr{U}_{b,\eta}$ in Section~\ref{3f}, and in Section~\ref{4f} we will explain the modifications needed to bound $\mathscr{U}^*_{b,\eta}$.

\subsubsection{Bounding $\mathscr{U}_{b,\eta}$: the case $(\ell_1\ell_1',\ell_2\ell_2')=1$}\label{3f}
We start by dividing $\mathscr{U}_{b,\eta}(M,N,A,\LL,\nu)$ into
\es{\label{stas}
\mathscr{U}_{b,\eta}=\mathscr{V}_{b,\eta}+\mathscr{V}^*_{b,\eta},
}
where $\mathscr{V}^*_{b,\eta}$ is the contribution of the terms such that  $d=d',\ell_1=\ell_1',\ell_2=\ell_2'$ and $a_1\neq a_1'$.

\paragraph{Bounding $\mathscr{V}^*_{b,\eta}$: the case $d=d',\ell_1=\ell_1',\ell_2=\ell_2',a_1\neq a_1'$}
In this section we deal with $\mathscr{V}^*_{b,\eta}$, which is given by
\est{
&\mathscr{V}^*_{b,\eta}:=
 \sumtwo_{\substack{p_1\p_2n_1', \q_1\q_2n_2'\in\Nb\\(b,n_1'n_2')=(n_1',\vartheta n_2')=1,\\ \mu^2(n_1')=1}}
 \sumtwo_{\substack{\ell_1,\ell_2\in\LL,\\ \p_1,\q_1|\ell_1,\ \p_2,\q_2|\ell_2\\(b\vartheta n_1'n_2',\ell_1\ell_2)=(\ell_1,\ell_2)=1,\\ \tilde \ell_1 \p_1 n_1' \neq \tilde \ell_2  \q_2n_2'}}
 \sum_{\substack{0\neq |d|\leq D/\q_1\p_2,\\(d,\tilde \ell_1   \tilde \ell_2  \p_1 \q_2n_1'n_2')=1,\\ n_2'\in \I}}
  \sumstar_{\substack{c\mod{b\q_1\p_2},\\\tilde \ell_1 \p_1 n_1' \equiv \tilde \ell_2  \q_2n_2'-cd\eta\\ \hspace{1.4em}\mod {b\q_1\p_2|d|\eta}  }}\\
 &\qquad\quad 
\sumthree_{\substack{a_1,a_1', a_2\in \A,\\ a_1\neq a_1'}}\nu_{a_1}\overline {\nu_{a_1'}}\e{\vartheta\prbigg{\overline c (a_1-a_1')\frac{\overline {\p_1n_1'}}{b\p_2}-\frac{d(a_1-a_1')\overline {\tilde \ell_2b\p_2   \q_2n_2'}}{\p_1n_1'}}}.\\
}
We reintroduce the complementary divisor in the congruence condition $\tilde \ell_1 \p_1 n_1' \linebreak[0]\equiv \tilde \ell_2  \q_2n_2'-cd\eta\mod {b\q_1\p_2|d|\eta}$ and reverse the previous computations. We then arrive at
\est{
\mathscr{V}^*_{b,\eta}&= \hspace{-0.87em}
 \sumtwo_{\substack{\p_1\p_2n_1',\q_1\q_2n_2'\in\Nb\\(b,n_1'n_2')=1,\\ (n_1',\vartheta n_2')=1,\\\mu^2(n_1')=1}}
\hspace{-0.51em} \sumtwo_{\substack{\ell_1,\ell_2\in\LL,\\ \p_1,\q_1|\ell_1,\ \p_2,\q_2|\ell_2\\(b\vartheta n_1'n_2',\ell_1\ell_2)=(\ell_1,\ell_2)=1,\\ \tilde \ell_1 \p_1 n_1' \neq \tilde \ell_2  \q_2n_2'}}
 \sumfour_{\substack{m\in\M,\,a_1,a_1',a_2\in\A,\ \eta|m\\(m, b\ell_1    \ell_2  n_1'n_2')=1,\ a_1\neq a_1'\\ \tilde \ell_1 \p_1 n_1' \equiv \tilde \ell_2  \q_2n_2'\mod m,}}
\hspace{-0.5em} \nu_{a_1}\overline {\nu_{a_1'}}\e{\vartheta\frac{(a_1-a_1')\overline {m}}{b\p_1\p_2n_1'}}.\\
}
Now, we introduce once again the complementary divisor, but this time we get rid of the variable $n_2'$: 
\est{
\mathscr{V}^*_{b,\eta}&=
 \sum_{\substack{\p_1\p_2n_1'\in \Nb,\\(b\vartheta,n_1')=1,\\\mu(n_1')^2=1}}
 \sumtwo_{\substack{\ell_1,\ell_2\in\LL,\\ \p_1,\q_1|\ell_1,\ \p_2,\q_2|\ell_2\\(b\vartheta n_1',\ell_1\ell_2)=(\ell_1,\ell_2)=1}}
\sum_{\substack{0\neq |d|\leq D/\q_1\p_2,\\(d,\tilde \ell_1   \tilde \ell_2  \p_1 \q_2n_1')=1}} 
\sumthree_{\substack{a_1,a_1', a_2\in \A,\\ a_1\neq a_1'}}\nu_{a_1}\overline {\nu_{a_1'}}
F(\cdots),\\
}
where
\est{
F(\cdots):= \sum_{\substack{m\in\M \cap\J,\ (m, b\ell_1    \ell_2  n_1')=1,\\ m\equiv\overline d \tilde \ell_1  \p_1n_1'\  (\tn{mod }\tilde \ell_2  \q_2) ,\  m\equiv0\mod\eta,\\ (md-\tilde \ell_1 \p_1 n_1' ,b\q_1 \ell_2 \q_2)=\tilde \ell_2 \q_2}}\e{\vartheta\frac{(a_1-a_1')\overline {m}}{b\p_1\p_2n_1'}},\\
}
for some interval $\J=\J(\ell_1, \ell_2 , \p_1,\p_2,\q_1,\q_2,d,n_2')$.
Notice that, in order not to lose the condition $(n_2',b\q_1\p_2)=1$, we have to introduce the condition $(md-\tilde \ell_1 \p_1 n_1' ,b\q_1\ell_2 \q_2)=\tilde \ell_2 \q_2$. 
 We apply~\eqref{kseq2} with $\beta=0$, $\chi$ the trivial character and
\est{
&k:=\eta\tilde{\ell_2}\q_2,\quad \gamma:=b\p_1\p_2n_1',\quad \gamma_1:=\frac{\gamma}{h_1}=\frac{b\p_1\p_2n_1'}{(  \q_2,\p_2)(\eta,n_1')},\quad c=b\q_1\ell_2 \q_2 \\
&h:=(k,\gamma)=(\eta,n_1')(\tilde \ell_2  \q_2,\p_2)=(\eta,n_1')(\q_2,\p_2),\qquad h_1:=(k^\infty,\gamma)=h,\\
}
where we used that $(\eta,b\ell_1\ell_2)=(n_1',\ell_2)=1$ and that $\gamma/b$ is square-free.
It follows that 
\est{
F(\cdots)
&\ll  (bN)^{\frac12+\eps}+(a_1-a_1',bn_1'\p_1\p_2)^{\frac12}\frac{M}{ \tilde{\ell_2} (b \p_1\p_2\q_2 n_1')^\frac12},
}
since $(\vartheta,bn_1'\ell_1\ell_2)=1$.
Dealing with the GCD as in~\eqref{adg}, we obtain
\as{
\mathscr{V}^*_{b,\eta} &\ll \|\nu\|^2  \frac{A^{2}DM^\eps}{\q_1\p_2}
 \sum_{\substack{\p_1\p_2n_1'\in \Nb}}
 \sumtwo_{\substack{\ell_1,\ell_2\in\LL,\\ \p_1,\q_1|\ell_1,\ \p_2,\q_2|\ell_2}}\pr{ (bN)^{\frac12}+\frac{M}{\tilde{\ell_2} (b \p_1\p_2\q_2 n_1')^\frac12}}\notag
\\
&\ll \|\nu\|^2  \frac{A^{2}DM^\eps}{\q_1\p_2}
 \sum_{\substack{\p_1\p_2n_1'\in \Nb}}
 \frac{L^2}{(\p_1+\q_1)(\p_2+\q_2)}\pr{ (bN)^{\frac12}+\frac{M}{ \tilde{\ell_2} (b \q_2 N)^\frac12 }}\notag
\\
&\ll  \frac{\|\nu\|^2A^{2}DL^2NM^\eps\pr{ \frac{(bN)^{\frac12}}{\p_1\p_2}+\frac{M}{ \p_1 L(b\q_2 N)^\frac12}}}{\q_1\p_2(\p_1+\q_1)(\p_2+\q_2)}\ll \frac{\|\nu\|^2b^{\frac12}L^2N^{\frac32}M^\eps DA^{2}}{\q_1\p_2(\p_1+\q_1)(\p_2+\q_2)},\label{1bd}
}
where we could assume without loss of generality that $1\leq D = 3\frac{LN}{ M}$, since otherwise the sum over $d$ in the definition of $\mathscr{V}^*_{b,\eta}$ is empty.

\paragraph{Bounding $\mathscr{V}_{b,\eta}$ : the case  $(d,\ell_1,\ell_2)\neq(d',\ell_1',\ell_2')$ if $a_1\neq a_1'$}\label{mccs}

Here we deal with $\mathscr{V}_{b,\eta}$, which consists of the terms of $\mathscr{U}_{b,\eta}$ such that $(d,\ell_1,\ell_2)\neq(d',\ell_1',\ell_2')$ if $a_1\neq a_1'$ (we remark that here $(\cdot,\cdot,\cdot)$ indicates a triple and not a GCD). Specifically,
\est{
&\mathscr{V}_{b,\eta}:=
 \sumtwo_{\substack{\p_1\p_2n_1', \q_1\q_2n_2'\in \Nb\\(b,n_1'n_2')=(n_1',\vartheta n_2')=1,\\ \mu^2(n_1')=1}}
 \sumfour_{\substack{\ell_1,\ell_2,\ell_1',\ell_2'\in\LL,\\ \p_1,\q_1|(\ell_1,\ell_1'),\ \p_2,\q_2|(\ell_2,\ell_2')\\(bn_1'n_2'\vartheta,\ell_1\ell_2\ell_1'\ell_2')=(\ell_1\ell_1',\ell_2\ell_2')=1,\\ \tilde \ell_1 \p_1 n_1' \neq \tilde \ell_2  \q_2n_2',\ \tilde \ell_1' \p_1 n_1' \neq \tilde \ell_2'  \q_2n_2'}}
 \sumtwo_{\substack{0\neq |d|,|d'|\leq D/\q_1\p_2,\\(d,\tilde \ell_1   \tilde \ell_2  \p_1 \q_2n_1'n_2')=1\\ (d',\tilde \ell_1'   \tilde \ell_2'  \p_1 \q_2n_1'n_2')=1\\ \ n_2'\in \I\cap\I'}}
\\
 & \hspace{2em} \sumstar_{\substack{c\mod{b\q_1\p_2},\\\tilde \ell_1 \p_1 n_1' \equiv \tilde \ell_2  \q_2n_2'-cd\eta\mod {b\q_1\p_2|d|\eta},\\ \tilde \ell_1' \p_1 n_1' \equiv \tilde \ell_2'  \q_2n_2'-cd'\eta\mod {b\q_1\p_2|d'|\eta}  }}
\sumthree_{\substack{a_1,a_1', a_2\in \A,\\ a_1\neq a_1'\Rightarrow (d,\ell_1,\ell_2)\neq(d',\ell_1',\ell_2')}}\nu_{a_1}\overline \nu_{a_1'}\e{\cdots},\\
}
with the argument of the exponential given by~\eqref{vha}.

We start the analysis of $\mathscr{V}_{b,\eta}$ by noticing that  the conditions
\est{
\tilde \ell_1 \p_1 n_1' -\tilde \ell_2  \q_2n_2'+cd\eta &\equiv 0\mod {b\q_1\p_2|d|\eta}\\
\tilde \ell_1' \p_1 n_1' -\tilde \ell_2'  \q_2n_2'+cd'\eta &\equiv 0\mod {b\q_1\p_2|d'|\eta}
}
imply the congruence conditions
\as{
(\tilde \ell_2' \tilde \ell_1-\tilde \ell_2\tilde \ell_1') \p_1 n_1'  +(\tilde \ell_2' d-\tilde \ell_2d')c\eta\equiv 0\mod {b\q_1\p_2\eta},\label{acafd}\\
(d'\tilde \ell_1 -d\tilde \ell_1') \p_1 n_1' \equiv (d'\tilde \ell_2 -d\tilde \ell_2' ) \q_2n_2' \mod {b}\notag
}
and thus also
\es{ \label{fja}
(d'\tilde \ell_1 -d\tilde \ell_1') \overline{\q_2n_2'} \equiv (d'\tilde \ell_2 -d\tilde \ell_2' ) \overline {\p_1n_1'} \mod {b},
}
since $(b,\p_1n_1'\q_2n_2')=1$.

Now, we use the congruence relation
\es{ \label{cond:3congruence}
\frac{\overline {\alpha\gamma}}{\beta}+\frac{\overline {\beta\gamma}}{\alpha}+\frac{\overline {\alpha\beta}}{\gamma}\equiv\frac{1}{\alpha\beta\gamma}\mod 1,
}
which holds for $\alpha,\beta,\gamma$ pairwise coprime, to rewrite $-\frac{a_2(d \tilde{ \ell_1'}-d' \tilde{ \ell_1})\overline{ \tilde \ell_1\tilde \ell_1'\p_1 n_1'  b\q_1}}{\q_2n_2'}$ (modulo 1) as 
\es{\label{cond:-congruence}
&\frac{a_2(d \tilde{ \ell_1'}-d' \tilde{ \ell_1})\overline{ \q_2 n_2'  b}}{\tilde \ell_1\tilde \ell_1'\p_1\q_1n_1'}+\frac{a_2(d \tilde{ \ell_1'}-d' \tilde{ \ell_1})\overline{ \tilde \ell_1\tilde \ell_1'\p_1 \q_1n_1'  \q_2n_2'}}{b}
-\frac{a_2(d \tilde{ \ell_1'}-d' \tilde{ \ell_1})}{b \tilde \ell_1\tilde \ell_1'\p_1 n_1'  \q_1\q_2n_2'}\\
&\equiv\frac{a_2(d \tilde{ \ell_1'}-d' \tilde{ \ell_1})\overline{ \q_2 n_2'  b}}{\tilde \ell_1\tilde \ell_1'\p_1\q_1n_1'}+\frac{a_2(d'\tilde \ell_2 -d\tilde \ell_2' )\overline{ \tilde \ell_1\tilde \ell_1'\p_1^2 \q_1n_1'  n_1'^2}}{b}
-\frac{a_2(d \tilde{ \ell_1'}-d' \tilde{ \ell_1})}{b \tilde \ell_1\tilde \ell_1'\p_1 n_1'  \q_1\q_2n_2'},
}
by~\eqref{fja}, 
and thus the argument~\eqref{vha} of the exponential becomes
\es{\label{fvw}
&
\vartheta \prbigg{\Delta\frac{ \overline {\tilde \ell_2\tilde \ell_2'b\p_2   \q_2n_2'}}{\tilde \ell_1\tilde \ell_1'\q_1 \p_1n_1'}
-\frac{a_2(d \tilde{ \ell_1'}-d' \tilde{ \ell_1})}{b \tilde \ell_1\tilde \ell_1'\p_1 n_1'  \q_1\q_2n_2'}+ (a_1-a_1')\frac{\overline {c\p_1n_1'}}{b\p_2}-\frac{a_2(d'\tilde \ell_2 -d\tilde \ell_2' )\overline{ \tilde \ell_1\tilde \ell_1'\p_1^2 \q_1n_1'^2  }}{b}},
}
where
\est{
\Delta:=a_2(d \tilde{ \ell_1'}-d' \tilde{ \ell_1})\tilde \ell_2\tilde \ell_2'\p_2-(da_1{\tilde \ell_2'}-d'a_1'{\tilde \ell_2})\tilde \ell_1\tilde \ell_1'\q_1.
}

We divide $\mathscr{V}_{b,\eta}$ into two parts, depending on whether $\Delta$ is equal to $0$ or not:
\es{\label{rbar}
\mathscr{V}_{b,\eta}=\mathscr{V}^{\Delta=0}_{b,\eta}+\mathscr{V}^{\Delta\neq 0}_{b,\eta}.
}
We shall give a trivial bound for the terms with $\Delta=0$, whereas we will use Weil's bound on the sum over $n_2'$ to handle the terms with $\Delta\neq0$.

\textbf{The terms with $\Delta=0$.}
The condition $\Delta=0$ gives that
\es{\label{bmr}
a_2(d \tilde{ \ell_1'}-d' \tilde{ \ell_1})\tilde \ell_2\tilde \ell_2'\p_2=(da_1{\tilde \ell_2'}-d'a_1'{\tilde \ell_2})\tilde \ell_1\tilde \ell_1'\q_1,
} 
and it allows us to express either $a_1$ or $a_1'$ in terms of all the other variables:
\est{
a_1 &= f(a_1',a_2,d,d',\ell_1,\ell_2, \ell_1', \ell_2',\p_1,\q_1,\p_2,\q_2),\\
a_1' &= g(a_1,a_2,d,d',\ell_1,\ell_2, \ell_1', \ell_2',\p_1,\q_1,\p_2,\q_2),
}
for some functions $f$ and $g$.
Moreover, since  $(\ell_1\ell_1', \ell_2\ell_2') =(d,\tilde \ell_1\tilde\ell_2)=(d',\tilde \ell_1'\tilde\ell_2')\linebreak[0]=1$, then  ~\eqref{bmr} implies that $\tilde \ell_1|a_2\tilde\ell_1'$, $\tilde \ell_1'|a_2\tilde\ell_1$ and $\tilde \ell_2|a_1\tilde\ell_2'$, $\tilde \ell_2'|a_1'\tilde\ell_2$, which, by definition of $\tilde\ell_i,\tilde\ell_i'$, is equivalent to $ \ell_1|a_2\ell_1'$, $ \ell_1'|a_2\ell_1$, $ \ell_2|a_1\ell_2'$, $ \ell_2'|a_1'\ell_2$. It follows that $\mathscr{V}^{\Delta=0}_{b,\eta}$ is bounded by
\as{
\mathscr{V}^{\Delta=0}_{b,\eta}& \ll
 \sumtwo_{\substack{\p_1\p_2n_1', \q_1\q_2n_2'\in \Nb}}
 \sumfour_{\substack{\ell_1,\ell_2,\ell_1',\ell_2'\in\LL,\\ \p_1,\q_1|(\ell_1,\ell_1'),\ \p_2,\q_2|(\ell_2,\ell_2'),\\\tilde \ell_1 \p_1 n_1' \neq \tilde \ell_2  \q_2n_2',\  \tilde \ell_1' \p_1 n_1' \neq \tilde \ell_2'  \q_2n_2'}}
 \sumtwo_{\substack{0\neq |d|,|d'|\leq \frac D{\q_1\p_2 },\\ d'|(\tilde \ell_1' \p_1 n_1' - \tilde \ell_2'  \q_2n_2') }}\notag\\
& 
 \sum_{\substack{c\mod{b\q_1\p_2},\\\tilde \ell_1 \p_1 n_1' \equiv \tilde \ell_2  \q_2n_2'-cd\eta,\\\hspace{2em}\mod {b\q_1\p_2|d|\eta}}}
\sumthree_{\substack{a_1,a_1', a_2\in \A,\\  \ell_1|a_2\ell_1',  \ell_1'|a_2\ell_1, \ell_2|a_1\ell_2',  \ell_2'|a_1'\ell_2,\\ a_1=f(\cdots),\ a_1'=g(\cdots) }}(|\nu_{a_1}|^2+ |\nu_{a_1'}|^2)\notag\\
& \ll M^\eps
 \sum_{\substack{\p_1\p_2n_1', \q_1\q_2n_2'\in \Nb}}
 \sumfour_{\substack{\ell_1,\ell_2,\ell_1',\ell_2'\in\LL,\\ \p_1,\q_1|(\ell_1,\ell_1'),\\ \p_2,\q_2|(\ell_2,\ell_2')}} \sumtwo_{\substack{a_1,a_2\in \A,\\  \ell_1|a_2\ell_1',\  \ell_1'|a_2\ell_1,\\ \ell_2|a_1\ell_2' }}|\nu_{a_1}|^2\notag\\
& \ll
 \|\nu\|^2\frac{AN^2M^{\eps}}{\p_1\p_2\q_1\q_2}
 \sumtwo_{\substack{\ell_1',\ell_2'\in\LL,\\ \p_1,\q_1|\ell_1',\  \p_2,\q_2|\ell_2'}} 1
 \ll\frac{ \|\nu\|^2AL^2N^{2}M^\eps}{\p_1\p_2\q_1\q_2(\p_1+\q_1)(\p_2+\q_2)},\label{2bd}
}
where we dropped the condition $(c,b\q_1\p_2)=1$ by positivity, and the sum over $c$ has the only effect of turning the congruence condition $\tilde \ell_1 \p_1 n_1' \equiv \tilde \ell_2  \q_2n_2'-cd\eta\mod {b\q_1\p_2|d|\eta}$ into $\tilde \ell_1 \p_1 n_1' \equiv \tilde \ell_2  \q_2n_2'\mod {|d|\eta}$, which then implies $d|(\tilde \ell_1 \p_1 n_1'-  \tilde \ell_2  \q_2n_2)$.

\textbf{The terms with $\Delta\neq0$.}
Here we bound $\mathscr{V}^{\Delta\neq0}_{b,\eta}$. Exchanging the order of summation and indicating with $G(\dots)$ the sum over $n_2'$, we see that
\es{\label{gwp}
\mathscr{V}^{\Delta\neq0}_{b,\eta}
& \ll
 \sum_{\substack{ \p_1\p_2n_1'\in \Nb,\\(b\vartheta ,n_1')=1,\\\mu^2(n_1')=1}}
 \sumfour_{\substack{\ell_1,\ell_2,\ell_1',\ell_2'\in\LL,\\ \p_1,\q_1|(\ell_1,\ell_1'),\ \p_2,\q_2|(\ell_2,\ell_2')\\(bn_1'\vartheta,\ell_1\ell_2\ell_1'\ell_2')=(\ell_1\ell_1',\ell_2\ell_2')=1}}
 \sumtwo_{\substack{0\neq |d|,|d'|\leq D/\q_1\p_2,\\(d,\tilde \ell_1   \tilde \ell_2  \p_1 \q_2n_1')=1\\ (d',\tilde \ell_1'   \tilde \ell_2'  \p_1 \q_2n_1')=1}}\\
 & 
 \sumstar_{\substack{c\mod{b\q_1\p_2},\\  (\tilde \ell_2' \tilde \ell_1-\tilde \ell_2\tilde \ell_1') \p_1 n_1' \equiv (\tilde \ell_2d'-\tilde \ell_2' d)c\eta\\ \hspace{5em}\mod {b\q_1\p_2\eta}}}
\sumthree_{\substack{a_1,a_1', a_2\in \A,\\ \Delta\neq 0,\\  (d, \ell_1, \ell_2)\neq (d', \ell_1', \ell_2')}}|\nu_{a_1} {\nu_{a_1'}}||G(\cdots)|,
}
where
\es{\label{hgfdev}
G(
\cdots):= \sum_{n_2',\, (*)}\e{\cdots},
}
and the argument of the exponential is
\es{\label{afed}
&\vartheta\Delta\frac{ \overline {\tilde \ell_2\tilde \ell_2'b\p_2   \q_2n_2'}}{\tilde \ell_1\tilde \ell_1'\q_1 \p_1n_1'}
-\frac{\vartheta a_2(d \tilde{ \ell_1'}-d' \tilde{ \ell_1})}{b \tilde \ell_1\tilde \ell_1'\p_1 n_1'  b\q_1\q_2n_2'}
\quad\mod1.\\
}
The condition $(*)$ indicates that $n_2'$ satisfies the following conditions:
\ast{
&n_2'\in \I\cap\I',\quad \q_1\q_2n_2'\in \Nb, &&\tilde \ell_1 \p_1 n_1' \equiv \tilde \ell_2  \q_2n_2'-cd\eta\mod {b\q_1\p_2|d|\eta},\\
&  (n_2',bdd'n_1'\ell_1\ell_2\ell_1'\ell_2')=1,&&
\tilde \ell_1' \p_1 n_1' \equiv \tilde \ell_2'  \q_2n_2'-cd'\eta\mod {b\q_1\p_2|d'|\eta},\\
&\tilde \ell_1 \p_1 n_1' \neq \tilde \ell_2  \q_2n_2', && \tilde \ell_1' \p_1 n_1' \neq \tilde \ell_2'  \q_2n_2'.
}
We remark that we could keep the condition~\eqref{acafd} and drop the last two summands of~\eqref{fvw}
  as they do not depend on $n_2'$. Also, notice that the condition  $(d,\ell_1,\ell_2)\neq(d',\ell_1',\ell_2')$ if $a_1\neq a_1'$ becomes simply $(d,\ell_1,\ell_2)\neq(d',\ell_1',\ell_2')$ since if these triples are equal then $\Delta\neq 0 $ implies $a_1\neq a_1'$.

We apply Lemma~\ref{ks} to the sum over $n_2'$, removing the second summand of~\eqref{afed} by using partial summation with the bound  
\est{
\vartheta \frac{a_2(d \tilde{ \ell_1'}-d' \tilde{ \ell_1})}{b \tilde \ell_1\tilde \ell_1'\p_1 n_1'  \q_1\q_2n_2'}\ll \frac{|\vartheta| ADL}{b \tilde \ell_1\tilde \ell_1'\q_1^2\p_1\p_2 n_1'  \q_1\q_2n_2'}\ll \frac{|\vartheta|AD}{bLN^2}.
}
Note that $n_2'$ runs over a finite union of intervals of length at most $O\pr{{N}/{\q_1\q_2}}$, with a congruence condition modulo $b\q_1\p_2[d,d']\eta$ (provided that the sum is non-empty), where $[a,b]$ is the LCM of $a$ and $b$. 

In the notation of Lemma~\ref{ks} (and under the various conditions of the sums in~\eqref{gwp}), we have $\gamma=\tilde \ell_1\tilde \ell_1'\q_1 \p_1n_1' $, $ k=b\q_1\p_2[d,d']\eta$ and
\ast{
& h=(\gamma,k)=\q_1([d,d'],\tilde \ell_1\tilde \ell_1')(\eta,n_1')
=\q_1(d,\tilde \ell_1')(d',\tilde\ell_1)(\eta,n_1'),\\
& h_1=(k^\infty,\gamma)=(\q_1^\infty[d,d']^\infty,\tilde \ell_1\tilde \ell_1' )(\eta,n_1')\q_1(\p_1,\q_1)=(\p_1,\q_1)h,
}
since $\q_1>1$ implies $\tilde \ell_1=\tilde \ell_1' =1$, whence
\ast{
&\gamma_1:=\frac{\gamma}{h_1}= \frac{\p_1}{(\p_1,\q_1)}\frac{n_1'}{(\eta,n_1')} \frac{\tilde \ell_1}{(d',\tilde \ell_1)}\frac{\tilde \ell_1'}{(d,\tilde \ell_1')}.
}
Thus, Lemma~\ref{ks} gives 
\ast{
G(\cdots)&\ll M^{\eps}\bigg((\p_1,\q_1)\bigg(\frac{N\tilde \ell_1\tilde\ell_1'}{(\p_1,\q_1)\p_2}\bigg)^{\frac12}+\frac{(\vartheta \Delta,\gamma_1)}{\gamma_1}\frac{N}{b\eta\q_1^2\q_2\p_2[d,d']}\bigg)\prBig{1+ \frac{|\vartheta|AD}{bLN^2}}\notag\\
&\ll M^{\eps}\pr{LN^{\frac12}+ \frac{(\Delta,n_1')\p_1}{b\q_1^2\q_2[d,d']}}\prBig{1+ \frac{|\vartheta|AD}{bLN^2}}.
}

Now, if we define
\est{
u:=(\tilde \ell_2' d-\tilde \ell_2d',b\q_1\p_2),\qquad v:=(\tilde \ell_2' d-\tilde \ell_2d')/u
}
(so that $(v,b\q_1\p_2/u)=1$), then by the congruence condition~\eqref{acafd}
we have 
\as{
&\tilde \ell_2' d\equiv \tilde \ell_2d'\mod u,\label{rtvd1}\\ 
&\p_1 \tilde \ell_1\equiv \p_1\overline{ \tilde \ell_2' } \tilde \ell_2\tilde \ell_1'\mod u\label{rtvd2}
}
since $(\tilde\ell_2',b\q_1\p_2)=1$. Also, if $v\neq0$, then
\es{\label{gfdae}
 c\equiv-\overline{ v\eta}\frac{(\tilde \ell_2' \tilde \ell_1-\tilde \ell_2\tilde \ell_1') \p_1 n_1'}{u}  \mod {\frac{b\q_1\p_2}{u}}.
}
We remark that we can assume $u\leq\frac{10DL}{\q_1\p_2}$. Indeed, if $u>\frac{10DL}{\q_1\p_2}>\tilde \ell_2' d+ \tilde \ell_2d' $ then~\eqref{rtvd1} implies $\tilde \ell_2' d=\tilde \ell_2d'$. Thus $u=b\q_1\p_2$ and $d=d'$, $\ell_2=\ell_2'$ since $(d,\tilde\ell_2)=(d',\tilde \ell_2')=1$. Moreover,~\eqref{rtvd2} would give $\p_1 \tilde \ell_1= \p_1\tilde \ell_1'$ (since we can assume $D/\q_1\p_2\geq1$) and thus $\ell_1=\ell_1'$.
So $(d, \ell_1, \ell_2)= (d', \ell_1', \ell_2')$, and these terms have been previously excluded.

Thus, if we bound trivially the sum over $c$ using~\eqref{gfdae} and drop some conditions by positivity
, we find
\as{
&\mathscr{V}^{\Delta\neq0}_{b,\eta}\ll
M^\eps \sum_{\substack{ \p_1\p_2n_1'\in \Nb}}
 \sumfour_{\substack{\ell_1,\ell_2,\ell_1',\ell_2'\in\LL,\\ \p_1,\q_1|(\ell_1,\ell_1'),\ \p_2,\q_2|(\ell_2,\ell_2'),\\(\ell_2',b\q_1\p_2)=1}}
 \sum_{\substack{u| b\q_1\p_2,\ u\leq \frac{10DL}{\q_1\p_2},\\ \p_1 \tilde \ell_1\equiv \p_1\overline{ \tilde \ell_2' } \tilde \ell_2\tilde \ell_1'\mod u}}
 \sumtwo_{\substack{ |d|,|d'|\leq D/\q_1\p_2,\\ d\equiv \overline{\tilde\ell_2'}\tilde \ell_2d'\mod u}} \notag\\
 &\hspace{4em} 
 \sumthree_{\substack{a_1,a_1', a_2\in \A,\\ \Delta\neq 0}}\,u\,|\nu_{a_1} {\nu_{a_1'}}|\pr{LN^{\frac12}+\frac{(\Delta,n_1')\p_1}{b\q_1^2\q_2[d,d']}}\prBig{1+ \frac{|\vartheta|AD}{bLN^2}} \notag\\
&\ll \|\nu\|^2 \frac{A^2LN^{\frac32}M^\eps}{\p_1\p_2}
 \sumfour_{\substack{\ell_1,\ell_2,\ell_1',\ell_2'\in\LL,\\ \p_1,\q_1|(\ell_1,\ell_1'),\\ \p_2,\q_2|(\ell_2,\ell_2'),\\(\ell_2',b\q_1\p_2)=1}}
 \sum_{\substack{u| b\q_1\p_2,\ u\leq \frac{10DL}{\q_1\p_2},\\ \p_1 \tilde \ell_1\equiv \p_1\overline{ \tilde \ell_2' } \tilde \ell_2\tilde \ell_1'\\\hspace{2em}\mod u}}\hspace{-0.6em}
 \sumtwo_{\substack{ |d|,|d'|\leq D/\q_1\p_2,\\ \hspace{0.8em}d\equiv \overline{\tilde\ell_2'}\tilde \ell_2d'\mod u}}
 \hspace{-0.5em}u\prBig{1+ \frac{|\vartheta|AD}{bLN^2}} \notag\\
&\ll \|\nu\|^2 \frac{A^2DL^{2}N^{\frac32}M^\eps}{(\p_1+\q_1)\p_1\q_1\p_2^2}
 \sumtwo_{\substack{\ell_2,\ell_2'\in\LL,\\ \p_2,\q_2|(\ell_2,\ell_2')}}
 \sum_{\substack{u| b\q_1\p_2,\\ u\leq \frac{10DL}{\q_1\p_2}}}\hspace{-0.2em}u\pr{\frac{L}{u}+1}
\pr{\frac{D}{u\q_1\p_2}+1}\prBig{1+ \frac{|\vartheta|AD}{bLN^2}}\notag \\
&\ll \|\nu\|^2 \frac{A^2DL^{2}N^{\frac32}M^\eps}{(\p_1+\q_1)\p_1\q_1\p_2^2}
\hspace{-0.1em} \sumtwo_{\substack{\ell_2,\ell_2'\in\LL,\\ \p_2,\q_2|(\ell_2,\ell_2')}}
\frac{DL}{\q_1\p_2}\prBig{1+ \frac{|\vartheta|AD}{bLN^2}}\notag \\
&\ll \frac{ \|\nu\|^2A^2D^2L^{5}N^{\frac32}M^\eps}{(\p_1+\q_1)(\p_2+\q_2)\p_1\q_1^2\p_2^3\q_2}\prBig{1+ \frac{|\vartheta|AD}{bLN^2}}.\label{boudn}
 }
\paragraph{Total bound for the terms with $(\ell_1\ell_1',\ell_2\ell_2')=1$}

By~\eqref{rbar},~\eqref{2bd} and~\eqref{boudn}, we obtain
\est{
\mathscr{V}_{b,\eta}\ll \|\nu\|^2M^\eps \prBig{1+ \frac{|\vartheta|AD}{bLN^2}}\frac{A^2D^2L^{5}N^{\frac32}+AL^2N^2 }{\q_1\p_2(\p_1+\q_1)(\p_2+\q_2)}
}
and thus, by~\eqref{stas} and~\eqref{1bd}, we have
\es{\label{1234}
\mathscr{U}_{b,\eta} &\ll \|\nu\|^2A^2L^{2}N^{\frac32}M^\eps
\frac{  (D  b^{\frac12}+L^3D^2+N^{\frac12}/ A)}{\q_1\p_2(\p_1+\q_1)(\p_2+\q_2)}\prBig{1+ \frac{|\vartheta|AD}{bLN^2}}.
}
\subsubsection{Bounding $\mathscr{U}^*_{b,\eta}$: the case $(\ell_1\ell_1',\ell_2\ell_2')>1$}\label{4f}
First, we observe that if $(\ell_1\ell_1',\ell_2\ell_2')>1$, then $(\ell_2,\ell_2')=(\ell_1,\ell_1')=1$ since we have $(\ell_1,\ell_2)=(\ell_1',\ell_2')=1$, and so $\p_1=\q_1=\p_2=\q_2=1.$ Thus, we can repeat the same arguments of Section~\ref{3f} in a slightly different but simplified form, and we obtain that the bound~\eqref{1234} holds also for $\mathscr{U}^*_{b,\eta}$. 

The only difference between this case and the $(\ell_1\ell_1',\ell_2\ell_2')=1$ case is that we cannot make the same choice of $\alpha,\beta,\gamma$
in~\eqref{cond:3congruence} as $\alpha,\beta,\gamma$ would not be pairwise coprime (and neither we could invert $\tilde \ell_2\tilde \ell_2' \ (\tn{mod }\tilde \ell_1\tilde \ell_1')$ in~\eqref{fvw}). To overcome this problem, it is enough to divide the sum over $n_2'$ into congruence classes modulo $(\ell_1\ell_1',\ell_2\ell_2')$ and apply~\eqref{cond:3congruence} with $\gamma=b(\ell_1\ell_1',\ell_2\ell_2')$ rather than $\gamma=b$ as in~\eqref{cond:-congruence}. The extra sum over the congruence classes modulo $(\ell_1\ell_1',\ell_2\ell_2')$ has the effect of making us lose a factor of $(\ell_1\ell_1',\ell_2\ell_2')$ in the $\Delta\neq0$ terms, but this loss is recovered by the extra condition between the $\ell_1,\ell_1',\ell_2,\ell_2'$. In fact, one can obtain a bound stronger than~\eqref{1234}  in this case, since the resulting Kloosterman has smaller modulus. 

\subsubsection{The final bound for $\mathscr{S}_{b,\eta}$}
Putting together~\eqref{aaebfr},~\eqref{eesbf} with the bound~\eqref{1234} and its analogue for $\mathscr{U}^*_{b,\eta}$, we obtain
\es{\label{ees}
\mathscr{S}_{b,\eta} & \ll \|\beta\|^2\|\nu\|^2 \Big(b+ \frac{|\vartheta|AD}{LN^2}\Big)^\frac12ALN^{\frac34}M^\eps\bigg(  D ^\frac12 b^{\frac14}+L^\frac32D+\frac{N^{\frac14}}{ A^\frac12}\bigg)\\
& \hspace{0em}\ll \|\beta\|^2\|\nu\|^2 \Big(b+ \frac{|\vartheta|A}{NM}\Big)^\frac12ALN^{\frac34}M^\eps\bigg(  \frac{b^{\frac14} N^\frac12 L^\frac12}{M^\frac12}+\frac{L^\frac52N}M+\frac{N^\frac14}{ A^\frac12}\bigg),
}
since $D=3\frac{NL}M$.

\subsection{The terms with $(n_1',n_2')>1$}

In this section we bound $\mathscr{S}^*_{b,\eta}(M,N,A,L; \beta,\nu)$, which consists of the sum \eqref{fafe} restricted to $(n_1',n_2')>1$.
We recall the definition~\eqref{vae1} of $n_1'$, $n_2'$:
\est{
n_1':=\frac{n_1}{(n_1,\ell_1\ell_2)},\quad n_2':=\frac{n_2}{(n_2,\ell_1\ell_2)}
}
We write $\mu=(n_1',n_2')$. This implies $(\mu,\ell_1\ell_2)=1$ and $n_1=\mu h_1$, $n_2=\mu h_2$. Thus, denoting $h_1':=\frac{h_1}{(h_1,\ell_1,\ell_2)}$, $h_2':=\frac{h_2}{(h_2,\ell_1,\ell_2)}$, we automatically have $(h_1',h_2')=(h_1'h_2',\ell_1\ell_2)=1$ since $n_1,n_2$ are square-free. It follows that 
\est{
\mathscr{S}^*_{b,\eta} &=\hspace{-0.4em}
 \sum_{\substack{\mu>1,\\ (\mu,\vartheta b)=1}}
 \sumseven_{\substack{\ell_1, \ell_2\in\LL,m\in\M,h_1 , h_2\in \Nb_\mu,a_1,a_2\in\A \\(mb\mu \vartheta,\ell_1\ell_2h_1h_2)=(m,\mu b)=(\ell_1,\ell_2)=(h_1',h_2')=1,
 \\\ell_1h_1\equiv \ell_2h_2\mod m,\ (h_1'h_2',\ell_1\ell_2)=1 \\ \ell_1 h_1\neq \ell_2h_2,\ \eta|m}}\hspace{-0.7em}
\beta_{\mu h_1}\nu_{a_1}\overline {\beta_{\mu h_2}\nu_{a_2}}\e{\frac{a_1\overline m}{b\mu h_1}-\frac{a_2\overline m}{b\mu h_2}}\\
&=\sum_{\substack{1<\mu\leq N,\\ (\mu,\vartheta b)=1}}\mathscr{S}_{b,\eta}(M,N/\mu,A,\LL, \beta_\mu,\nu), 
}
where $\beta_\mu(n):=\beta_{\mu n}$ and $\Nb_\mu=[N/(2\mu),N/\mu]$. Thus, by~\eqref{ees}
\as{
\mathscr{S}^*_{b,\eta} &\ll \|\nu\|^2 \prBig{b+ \frac{|\vartheta|A}{NM}}^\frac12ALN^{\frac34}M^\eps\sum_{\mu\leq N}\frac{\|\beta_\mu\|^2}{\mu^{\frac14}}\prbigg{  \frac{b^\frac14 N^\frac12 L^\frac12}{\mu^\frac14 M^\frac12}+\frac{L^\frac52N}{\mu M}+\frac{N^\frac14}{\mu^\frac14 A^\frac12}}\notag\\
& \ll \|\beta\|^2\|\nu\|^2 \prBig{b+ \frac{|\vartheta|A}{NM}}^\frac12ALN^{\frac34}M^\eps\prbigg{  \frac{b^\frac14 N^\frac12 L^\frac12}{M^\frac12}+\frac{L^\frac52N}M+\frac{N^\frac14}{ A^\frac12}}.\label{rtyu}
}

\subsection{The final bound for the off-diagonal term}
By~\eqref{dft} and the bounds~\eqref{ees} and~\eqref{rtyu}, we have
\est{
\mathscr{E}_{b,\eta}& \ll \|\beta\|^2\|\nu\|^2 \prBig{b+ \frac{|\vartheta|A}{NM}}^\frac12ALN^{\frac34}M^\eps\bigg(  \frac{b^\frac14 N^\frac12 L^\frac12}{M^\frac12}+\frac{L^\frac52N}M+\frac{N^\frac14}{ A^\frac12}\bigg).
}
Thus, by~\eqref{tqv}
\est{
\mathscr{E}_{b,1}^* \ll \|\beta\|^2\|\nu\|^2 \prBig{b+ \frac{|\vartheta|A}{NM}}^\frac12ALN^{\frac34}M^\eps\bigg(  \frac{b^\frac14 N^\frac12 }{M^\frac12}+\frac{N}M+\frac{N^\frac14}{ A^\frac12}\bigg)
}
and so, by~\eqref{gyui},
\es{\label{gfdd}
 \mathscr{O}_{b} & \ll \|\beta\|^2\|\nu\|^2 \prBig{b+ \frac{|\vartheta|A}{NM}}^\frac12ALN^{\frac34}M^\eps\bigg(  \frac{b^\frac14 N^\frac12 L^\frac12}{M^\frac12}+\frac{L^\frac52N}M+\frac{N^\frac14}{ A^\frac12}\bigg).
}

\section{Optimizing the parameter $L$}\label{optl}
Combining~\eqref{def:DbEb} with the bounds for the diagonal~\eqref{eqn:diagonal_ellnequal} and off-diagonal terms~\eqref{gfdd} we obtain
\est{
\mathcal{D}_b&\ll \|\beta\|^2\|\nu\|^2  \prBig{1+ \frac{|\vartheta|A}{bNM}}^\frac12 LM^\eps\bigg(A(bLN)^{\frac 12}+\frac{AM}{bN}+M\\
&\hspace{10em}+ \frac{b^\frac34A N^\frac54 L^\frac12}{M^\frac12}+\frac{b^\frac12AL^\frac52N^\frac74}M+b^\frac12 A^\frac12N\bigg)
}
and thus, by~\eqref{amp},
\es{\label{rve}
\mathcal{C}_b &\ll \|\beta\|^2\|\nu\|^2 M^\eps \prBig{1+ \frac{|\vartheta|A}{bNM}}^\frac12 \bigg(\frac{AM(bN)^{\frac 12}}{L^\frac12}+\frac{AM^2}{bLN}+\frac{M^2}L\\
&\hspace{10em}+ \frac{b^\frac34A M^\frac12 N^\frac54}{ L^\frac12}+b^\frac12AL^\frac32N^\frac74+\frac{b^\frac12 A^\frac12MN}L\bigg).
}
We wish to choose $L$ so that 
\est{
b^\frac12AL^\frac32N^\frac74\approx\frac{AM^2}{bLN}+\frac{M^2}L+\frac{b^\frac12 A^\frac12MN}L, 
}
and $L\geq2\log (b\vartheta M)$. So we take
\est{
L=\frac{M^{\frac45}}{b^\frac35 N^{\frac{11}{10}}}+\frac{M^\frac45}{b^\frac15 A^\frac25N^{\frac7{10}}}+\frac{ M^\frac25}{A^\frac15N^\frac3{10}}+M^\eps.
}
With this choice~\eqref{rve} implies
\as{
\mathcal{C}_b &\ll \|\beta\|^2\|\nu\|^2M^\eps \prBig{1+ \frac{|\vartheta|A}{bNM}}^\frac12 \bigg(AM(bN)^{\frac 12}+ b^\frac34A M^\frac12 N^\frac54\notag\\
&\qquad+b^\frac12AN^\frac74\bigg(\frac{M^{\frac45}}{b^\frac35 N^{\frac{11}{10}}}+\frac{M^\frac45}{b^\frac15 A^\frac25N^{\frac7{10}}}+\frac{ M^\frac25}{A^\frac15N^\frac3{10}}+M^\eps\bigg)^\frac32\bigg)\notag\\
&\ll \|\beta\|^2\|\nu\|^2 M^\eps \prBig{1+ \frac{|\vartheta|A}{bNM}}^\frac12 \bigg(AM(bN)^{\frac 12}+ b^\frac34A M^\frac12 N^\frac54\notag\\
&\qquad+\frac{AM^{\frac65}N^\frac1{10}}{b^\frac2{5} }+b^\frac1{5}A^\frac25 M^\frac65N^\frac7{10}+A^\frac7{10} b^\frac12M^\frac35N^\frac{13}{10}+b^\frac12AN^\frac74\bigg).\label{rveta}
}

\section{Removing the square-free condition}\label{rsq}

We write $n=bn'$, where $n'$ is square-free, $b$ is square-full, and $(b,n')=1$. We have
\est{
&\bigg|\sum_{\substack{a\in \A,}}\sum_{\substack{n\in \Nb,\\(m,n)=1}}\beta_n\nu_a\e{\vartheta a\frac{\overline m}{n}}\bigg|^2=\bigg|\sum_{\substack{b\leq N,\\ (b,\vartheta )=1,\\ b\tn{ square-full}}}\sum_{\substack{a\in \A,}}\sum_{\substack{bn'\in \Nb,\\(b\vartheta m,n')=1,\\\ n'\tn{square-free}}}\beta_{b n'}\nu_a\e{\vartheta a\frac{\overline m}{b n'}}\bigg|^2\\
&\hspace{2em}\ll\bigg(\sum_{\substack{b\leq N,\\ b\tn{ square-full}}}\hspace{-0.2em}b^{-\frac12}\bigg) \bigg(\sum_{\substack{b\leq N,\ (b,\vartheta )=1\\ b\tn{ square-full}}}b^{\frac12}\bigg|\sum_{\substack{a\in \A,}}\sum_{\substack{bn'\in \Nb,\\(b\vartheta  m,n')=1}}\mu(n')^2\beta_{b n'}\nu_a\e{\vartheta  a\frac{\overline m}{b n'}}\bigg|^2\bigg)\\
&\hspace{2em}\ll\sum_{\substack{b\leq N,\ (b,\vartheta )=1\\ b\tn{ square-full}}}b^{\frac12}\,\bigg|\sum_{\substack{a\in \A,}}\sum_{\substack{bn'\in \Nb,\\(b\vartheta  m,n')=1}}\mu(n')^2\beta_{b n'}\nu_a\e{\vartheta  a\frac{\overline m}{b n'}}\bigg|^2 M^\eps.
}
Thus,
\est{
\mathcal{C}_1(M,N,A;\beta,\nu)&=\sum_{\substack{m\in \M}}\bigg|\sum_{\substack{a\in \A,}}\sum_{\substack{n\in\Nb,\\(m,n)=1}}\beta_n\nu_a\e{a\frac{\overline m}{n}}\bigg|^2\\
&\ll\hspace{-0.4em} \sum_{\substack{b\leq N,\ (b,\vartheta )=1,\\ b\tn{ square-full}}}\hspace{-0.4em}b^{\frac12}\sum_{\substack{m\in \M,\\(m,b)=1}}\bigg|\sum_{\substack{a\in \A}}\sum_{\substack{n'\in \Nb_b,\\(b\vartheta  m,n')=1}}\mu(n')^2\beta_{b n'}\nu_a\e{\vartheta  a\frac{\overline m}{b n'}}\bigg|^2M^\eps\\
&= M^\eps\sum_{\substack{b\leq N,\ (b,\vartheta )=1\\ b\tn{ square-full}}}b^{\frac12}\mathcal{C}_b(M,N/b,A;\beta_b,\nu) \\
}
where $\beta_b(n):=\mu(n)^2\beta_{bn}$ and $\Nb_b:=[N/(2b),N/b]$. Now, if $b \leq B$, we apply~\eqref{rveta} and obtain that the contribution of those terms to $\mathcal{C}_1$ is bounded by
\es{\label{rvetad}
& \ll \|\nu\|^2M^\eps \prBig{1+ \frac{|\vartheta |A}{NM}}^\frac12 \sum_{\substack{b\leq B,\\ b\tn{ square-full}}}b^{\frac12}  \|\beta_b\|^2\bigg(AMN^{\frac 12}+ \frac{A M^\frac12 N^\frac54}{b^\frac12}\\
&\hspace{7em}+\frac{AM^{\frac65}N^\frac1{10}}{b^\frac12 }+\frac{A^\frac25 M^\frac65N^\frac7{10}}{b^\frac12}+\frac{A^\frac7{10} M^\frac35N^\frac{13}{10}}{b^\frac45}+\frac{AN^\frac74}{b^\frac54}\bigg)\\
& \ll \|\beta\|^2\|\nu\|^2M^\eps \prBig{1+ \frac{|\vartheta |A}{NM}}^\frac12 \bigg(AMN^{\frac 12}B^\frac12+ A M^\frac12 N^\frac54+AM^{\frac65}N^\frac1{10}+\\
&\hspace{13em}+A^\frac25 M^\frac65N^\frac7{10}+A^\frac7{10} M^\frac35N^\frac{13}{10}+AN^\frac74\bigg).\\
}
For $b> B$ we apply the trivial bound $\mathcal{C}_b(M,N/b,A,\beta_b,\nu)\ll \|\nu\|^2 A\frac{MN}b\|\beta_b\|$ and get that these terms contribute
\es{\label{fdew}
\ll \|\nu\|^2 ANM^{1+\eps} \sum_{\substack{b>B,\\ b\tn{ square-full}}}\|\beta_b\|^2 \frac{1}{b^{\frac12}} \ll  \|\beta\|^2\|\nu\|^2 \frac{ANM^{1+\eps}}{B^{\frac12}}.
}
We choose $B=N^{\frac12}$, so that combining~\eqref{fdew} and~\eqref{rvetad} we obtain
\est{
 \mathcal{C}_1&\ll \|\beta\|^2\|\nu\|^2M^\eps \prBig{1+ \frac{|\vartheta |A}{NM}}^\frac12 \\
&\times\hspace{0em}\big(AMN^{\frac 34}+ A M^\frac12 N^\frac54+AM^{\frac65}N^\frac1{10}+A^\frac25 M^\frac65N^\frac7{10}+A^\frac7{10} M^\frac35N^\frac{13}{10}+AN^\frac74\big).\\
}
Notice that the third summand on the second line can be absorbed. Indeed, if $M\ll N^2$, then $AM^{\frac65}N^\frac1{10}\ll AMN^{\frac34}$, whereas if $M\gg N^2$ then one can obtain a stronger bound from Theorem 5 of~\cite{DFI97}, which gives 
\es{\label{gfdscreb}
 \mathcal{C}_1&\ll \|\beta\|^2\|\nu\|^2AM^{1+\eps}
}
in such range. Moreover, we also have $AM^\frac12N^\frac54\ll AMN^\frac34+AN^\frac74$ and thus
\es{\label{ryx}
 \mathcal{C}_1&\ll \|\beta\|^2\|\nu\|^2(MN)^\eps \prBig{1+ \frac{|\vartheta |A}{NM}}^{\frac12}\\
 &\quad\times \big(AMN^{\frac 34}+AN^\frac74+A^\frac25 M^\frac65N^\frac7{10}+A^\frac7{10} M^\frac35N^\frac{13}{10}\big),
}
and this bound holds also without the assumption $\vartheta ,A,N\ll M^{C}$, for some $C>0$, since it is trivial otherwise.

\section{Completion of the proof of Theorem~\ref{thm:boundTri}}\label{rsq3}

Combining the bound~\eqref{ryx} for $\mathcal{C}_1$ and~\eqref{bfc}, we deduce
\es{\label{unconditional}
 \mathcal{B}(M,N,A)&\ll \|\alpha\|\|\beta\|\|\nu\|(AMN)^\eps \prBig{1+ \frac{|\vartheta |A}{NM}}^{\frac14}\\
 &\quad \big(A^\frac12M^\frac12N^{\frac 38}+A^\frac12N^\frac78+A^\frac15 M^\frac35N^\frac7{20}+A^\frac7{20} M^\frac3{10}N^\frac{13}{20}\big).
}
If $M\geq N$ this implies
\est{
 \mathcal{B}(M,N,A)&\ll \|\alpha\|\|\beta\|\|\nu\|(MN)^\eps \prBig{1+\frac{|\vartheta |A}{MN}}^\frac14\big(A^\frac12M^\frac12N^{\frac 38}+A^\frac7{20} M^\frac35N^\frac7{20}\big),
}
which is Theorem~\ref{thm:boundTri} in the range $M\geq N$. We then use the following remark to obtain~\eqref{mrthm:boundTri} also in the range $M < N$.

\begin{remark}
All the computations of Section~\ref{sec:diagonal_ln_equald} and~\ref{sec:offdiagonal:initial} work, applying partial summation at appropriate places,  also when an extra addend $f_{a, \vartheta}(m,bn)$ (with $f_{a, \vartheta}(x,y)$ satisfying \eqref{tckgf}) is inserted in the exponential function in the definition~\eqref{amp} of $\mathcal C_b$. Thus, one arrives at the bound 
\es{\label{gfd}
\mathcal B(f;M,N,A)&\ll \|\alpha\|\|\beta\|\|\nu\|(MN)^\eps \Big(1+\frac{|\vartheta |A+X}{MN}\Big)^\frac{1}{2}\\
&\quad\times\big(A^\frac12M^\frac12N^{\frac 38}+A^\frac7{20} M^\frac35N^\frac7{20}\big)
}
in the case $M\geq N$. Notice that the exponent of $ (1+\frac{|\vartheta |A+X}{MN})$ is now $\frac{1}{2}$ because in this case we need to apply partial summation also when dealing with the diagonal term and for the 
analogue of~\eqref{gfdscreb}.
\end{remark}
We now observe that the elementary reciprocity law allows us to write $\mathcal{B}(M,N,A)$ as
\est{
\mathcal{B}(M,N,A)=\sumthree_{\substack{m\in\M ,n\in\Nb,a\in\A ,\\(m,n)=1}}\alpha_{m}\beta_n\nu_a\e{-\vartheta \frac{a\overline n}{m}+\frac{\vartheta  a}{mn}}.
}
Thus, if $M<N$ we can apply~\eqref{gfd} with the role of $M$ and $N$ switched and with $f_{a,\vartheta}(x,y):=\frac{\vartheta a}{xy}$. Since $f_{a,\vartheta}(x,y)$ satisfies the condition~\eqref{tckgf} with $X=|\vartheta |A$, we obtain~\eqref{mrthm:boundTri} for $M < N$, and so the proof of Theorem~\ref{thm:boundTri} is complete.

\section{Proof of Theorem~\ref{thtw}}\label{wec}

We proceed in the same way as in the proof of Theorem~\ref{thm:boundTri}, with very minor modifications. First, we assume $bN\geq M$ and in the off-diagonal term we split the sum modulo $m$ into classes $m\equiv c\mod {4b\q_1\p_2}$ with $(m,b\q_1\p_2)=1$ rather than modulo $b\q_1\p_2$. Notice that this implies $ \tilde \ell_2  \q_2n_2' \equiv\tilde \ell_1 \p_1 n_1' -cd\mod {4d}$. Then, we can repeat the same arguments keeping the Jacobi symbol $(\frac{m}{n_1n_2})$ inside the summations, so that in the analogue of Section~\ref{gfdsa} this factor will become
 \est{
\pr{\frac{m}{n_1n_2}}&=\pr{\frac{c}{\q_1\p_2}}\pr{\frac{m}{\p_1\q_2n_1'n_2'}}=\pr{\frac{c}{\q_1\p_2}}\pr{\frac{d}{\p_1\q_2n_1'n_2'}}\pr{\frac{dm}{\p_1\q_2n_1'n_2'}}\\
&=\bigg(\frac{c}{\q_1\p_2}\bigg)\bigg(\frac{d}{\p_1\q_2n_1'}\bigg)\bigg(\frac{d}{\tilde \ell_2\q_2n_2'}\bigg)\bigg(\frac{d}{\tilde \ell_2\q_2}\bigg)\bigg(\frac{  -\tilde \ell_2  \q_2n_2'}{\p_1n_1'}\bigg)\bigg(\frac{ \tilde \ell_1 \p_1 n_1'}{\q_2n_2'}\bigg)\\
&=f\pr{c,\p_1,\p_2,\q_1,\q_2,n_1',n_2'} g\pr{c,\p_1,\p_2,\q_1,\q_2,n_1',\ell_1,\ell_2}\bigg(\frac{ \tilde \ell_1}{n_2'}\bigg),\\
}
where
\est{
f\pr{\dots}&:=\pr{\frac{c}{\q_1\p_2}}\pr{\frac{  -  \q_2n_2'}{\p_1n_1'}}\pr{\frac{  \p_1 n_1'}{\q_2n_2'}},\\
g\pr{\dots}&:=\bigg(\frac{d}{\p_1\q_2n_1'}\bigg)\bigg(\frac{d}{\tilde \ell_1 \p_1 n_1' -cd}\bigg)\bigg(\frac{d}{\tilde \ell_2\q_2}\bigg)\bigg(\frac{  \tilde \ell_2 }{\p_1n_1'}\bigg)\bigg(\frac{ \tilde \ell_1}{\q_2}\bigg),
}
and where we used that $\pr{\frac{a}{b}}$ is periodic modulo $b$ in $a$ if $b\geq1$ and is periodic modulo $4|a|$ in $b$ if $a\neq0$.

Thus, we arrive at the analogue of~\eqref{hgfdev} which will include an extra factor of $(\frac{ \tilde \ell_1}{n_2'})$. The rest of the argument is identical, with the difference that we apply~\eqref{kseq2} rather than~\eqref{kseq}. The fact that the former bound is weaker than the latter does not effect the arguments when $bN\geq M$. Thus, we obtain~\eqref{unconditional} and whence~\eqref{acaavcf} when $N\geq M$. This also implies~\eqref{acaavcf} for $M>N$, as can be seen by using the reciprocity formulae for the Jacobi symbol and for $\frac{\overline m}{n}$ and applying the bound for the case $N\geq M$.

\section{Proof of Corollary~\ref{c1}}\label{pc1}
Proceeding as in~\cite{DFI95}, we see that the error term in~\eqref{repd} is equal to
\est{
&\sum_{d|\Delta}\int_0^{4D}\hspace{-1.2em}\sumthree_{\substack{dn_1\in\Nb_1,dn_2\in\Nb_2,h\in\Z\\ 1\leq |h|\leq HD^\eps,\,(n_1,n_2)=1}}\hspace{-0.3em}\!\frac{\tilde\alpha_{dn_1}(x)\tilde\beta_{dn_2}(x+\Delta)}{dn_1n_2}\e{\frac{h\Delta}{d}\frac{\overline n_1}{n_2}}\!\e{\frac{hx}{dn_1n_2}}\,dx
+O(\eta),
}
with $H=\frac{\eta}{d}(\frac{N_1}{M_1}+\frac{N_2}{M_2})$, $D=M_1N_2+M_2N_1$ and $\alpha_{r}(x):=\alpha_{r}g\pr{ x/{r}}$, $\tilde\beta_{r}:=\beta_{r} f\pr{{x}/{r}}$. Applying Theorem~\ref{thm:boundTri} (in the version given by Remark~\ref{ptrmk}) we see that we can bound the above sum by
\est{
&\ll \sum_{d|\Delta}\sum_{1\leq |h|\leq HD^\eps} \frac{dD}{N_1N_2}\frac{\|\alpha\| \|\beta\|}{d^{\frac{1}{2}}} \Big(1+\frac{dHD}{N_1N_2}\Big)^\frac{1}{2}(N_1N_2)^{\frac7{20}}(N_1+N_2)^{\frac14+\eps},\\
&\ll \eta^{\frac{3}{2}}\pr{\frac{M_1N_2}{M_2N_1}+\frac{M_2N_1}{M_1N_2}}^\frac{3}{2}\|\alpha\|\|\beta\|(N_1N_2)^{\frac7{20}}(N_1+N_2)^{\frac14+\eps}(M_1M_2)^\eps,\\
}
since we can assume $|\Delta|\leq 4D$ (otherwise both $\mathcal{T}$ and the main term on the right hand side of~\eqref{repd} are identically zero) and since 
\est{
\frac{dHD}{N_1N_2}={\eta}\pr{\frac{N_1}{M_1}+\frac{N_2}{M_2}} \pr{\frac{M_1}{N_1}+\frac{M_2}{N_2}}\leq2 \eta\pr{\frac{M_1N_2}{M_2N_1}+\frac{M_2N_1}{M_1N_2}}.
}

\appendix

\section{Appendix: Weil's bound for incomplete Kloosterman sums}

In this appendix we give the following bounds for incomplete Kloosterman sums where the sums are subject to some coprimality and congruence conditions.

\begin{lemma}\label{ks}
Let $\delta,k,\gamma\geq1$, $\alpha,a,b,v\in\Z$ and let $I=\{x\in [X',X'+X], x\equiv v \mod k\}$. Let $h:=(k,\gamma)$, $h_1:=(k^\infty, \gamma)$, $\gamma_1:=\frac \gamma {h_1}  $, where  $(m^\infty,n):=\lim_{r\rightarrow\infty}(m^r,n)$. Then,
\es{\label{kseq}
\sum_{\substack{x\in I,\\(x,\gamma\delta)=1}}\e{\frac{\alpha\overline x}{\gamma}}\ll (\gamma\delta)^\eps\frac{h_1}{h}\pr{\frac{\gamma_1}{(\alpha,\gamma_1)}}^{\frac1{2}}+(\alpha,\gamma_1)\frac{X\delta^\eps}{\gamma_1 k}.
}
Moreover, if $\chi$ is a Dirichlet character modulo $\gamma$ and $c,d\geq1$, $a,b,\beta\in\Z$, then
\es{\label{kseq2}
\sum_{\substack{x\in I,\\(x,\gamma\delta)=1,\\ (ax+b,c)=d}}\!\chi(x)\e{\frac{\alpha\overline x+\beta x}{\gamma}}\ll (c\gamma\delta)^\eps\frac{h_1}{h}\pr{\frac{\gamma_1}{(\alpha,\gamma_1)}}^{\frac1{2}}+(\alpha,\gamma_1)^{\frac12}\gamma_1^{\frac12+\eps}\frac{X(c\delta)^\eps}{\gamma_1 k}.
}
\end{lemma}
\begin{proof}
We start by proving~\eqref{kseq}. 
First of all, we notice that we can assume that $(\delta,\gamma k)=(v,h)=1$ and $k\leq X$. Also, the case $\delta>1$ can be easily obtained from the case $\delta=1$ by M\"obius inversion. Indeed
\est{
\sum_{\substack{x\in I,\\(x,\gamma\delta)=1}}\e{\frac{\alpha\overline x}{\gamma}}&=\sum_{\substack{r_1|\delta}}\mu(r_1){\sum_{\substack{x\in I,\ (x,\gamma)=1,\\ x\equiv 0\mod {r_1}}}\e{\frac{\alpha\overline x}{\gamma}}}\\
&\ll \sum_{d|\delta}\pr{\gamma^\eps \frac{h_1}{h}\pr{\frac{\gamma_1}{(\alpha,\gamma_1)}}^{\frac1{2}+\varepsilon}+(\alpha,\gamma_1)\frac{X}{\gamma_1 k}}\\
&\ll (\gamma \delta)^\eps \frac{h_1}{h}\pr{\frac{\gamma_1}{(\alpha,\gamma_1)}}^{\frac1{2}}+\frac{(\alpha,\gamma_1)X\delta^\eps}{\gamma_1k},
}
by~\eqref{kseq} in the case $\delta=1$. Similarly, the case $(k,\gamma)>1$ can be recovered from the case $h=(k,\gamma)=1$:
\est{
&\sum_{\substack{x\in I,\\(x,\gamma)=1}}\!\e{\frac{\alpha\overline x}{\gamma}}=\hspace{-0.25em}\!\sumstar_{\substack{c\mod {h_1},\\ c\equiv v\,(\tn{mod}\, h)} }\sum_{\substack{x\in I_c,\\(x,\gamma_1)=1}}\!\e{\frac{\alpha\overline x}{\gamma}}=\hspace{-0.45em}\sumstar_{\substack{c\mod {h_1},\\ c\equiv v\,(\tn{mod}\, h)} }\!\e{\frac{\alpha\overline {c\gamma_1}}{h_1}}\!\!\sum_{\substack{x\in I_c,\\(x,\gamma_1)=1}}\!\e{\frac{\alpha\overline {h_1x}}{\gamma_1}}\\
&\hspace{3em}\ll\sumstar_{\substack{c\mod {h_1},\\ c\equiv v\mod h } }\bigg(\frac{\gamma_1^{\frac12+\eps}}{(\alpha,\gamma_1)^{\frac12}}+\frac{(\alpha,\gamma_1)X}{\gamma_1 k\frac{h_1}{h}}\bigg)\ll \delta ^\eps\frac{h_1}{h}\frac{\gamma_1^{\frac12+\eps}}{(\alpha,\gamma_1)^{\frac12}}+\delta^\eps(\alpha,\gamma_1)\frac{X}{\gamma_1 k},
}
where $I_c:=I\cap \{x\equiv c \mod {h_1} \}$. Thus, we can assume $h=1$ and similarly also that $(\alpha,\gamma)=1$. 
Now, applying the Erd\"os-Tur\'an inequality as in Lemma~8 of~\cite{DFI97}, we find
\est{
\bigg|{\sum_{\substack{x\in I,\\(x,\gamma)=1}}\e{\frac{\alpha\overline x}{\gamma}}}\bigg|\leq \frac{X+k}{\gamma k}|S(\alpha,0;\gamma)|+\sum_{1\leq r\leq\frac\gamma2}\frac1r|S(\alpha,r\overline k;\gamma)|.
}
Thus, using Weil's bound for $|S(\alpha,r\overline k;\gamma)|$ and observing that $|S(\alpha,0;\gamma)|$ is a Ramanujan sum and thus is bounded by $(\alpha,\gamma)$, we obtain~\eqref{kseq}. 

To prove~\eqref{kseq2} we can proceed in a similar way. As before, we assume $(\delta,\gamma k)=(v,h)=1$, $k\leq X$ and $\delta=1$. Also, we can assume $d|c$, $(a,c)=1$.

Moreover, we can deduce the case $c>1$ from the case $c=1$ as the following. For the case $\delta>1$, we start by using M\"obius inversion and we obtain
\es{\label{aarf}
\sum_{\substack{x\in I,\\(x,\gamma)=1,\\ (ax+b,c)=d}}\!\chi(x)\e{\frac{\alpha\overline x+\beta x}{\gamma}}&=\sum_{\substack{r|c,\\ d|r}}\mu(r)\sum_{\substack{x\in I,\ (x,\gamma)=1,\\ x\equiv -\overline{a}b\mod{r}}}\hspace{-0.6em}\chi(x)\e{\frac{\alpha\overline x+\beta x}{\gamma}}.\\
}
Next, we let $r=g_1g_2$, where 
$g_1$ is the largest divisor of $r$ such that $(\gamma/(\gamma,g_1),g_1)=1$. Also, we write $\gamma=f_1f_2$ where $f_1=(\gamma,g_1)$ and notice that $(g_1,g_2)=(f_1,f_2)=1$ and $f_1|g_1$, $g_2|f_2$ (and thus also $(f_1,g_2)=(g_1,f_2)=1$). Also,  we can write $\chi$ as $\chi_1\chi_2$, where for $i=1,2$, $\chi_i$ is a Dirichlet character modulo $f_i$. Using the orthogonality of additive character we obtain that the left hand side of~\eqref{aarf} is equal to
\est{
&\sum_{\substack{r|c,\\ d|r}}\frac{\mu(r)}{g_2}\sum_{u\mod {g_2}}\e{\frac{u\overline a b}{g_2}}\sum_{\substack{x\in I,\ (x,\gamma)=1,\\ x\equiv -\overline{a}b\mod{g_1}}}\!\chi(x)\e{\frac{\alpha\overline x+\beta x}{\gamma}+\frac{ux}{g_2}}\\
&\leq \sum_{\substack{r|c,\\ d|r}}\frac{1}{g_2}\sum_{u\mod {g_2}}\bigg|\sum_{\substack{x\in I,\ (x,\gamma)=1,\\ x\equiv -\overline{a}b\mod{g_1}}}\!\chi_2(x)\e{\frac{\alpha\overline {f_1x}+(\beta \overline {f_1}+uf_2/g_2)x}{f_2}}\bigg|,\\
}
since $x \equiv -\overline{a}b \mod {f_1}$ and thus the Chinese remainder theorem gives
\est{
\chi(x)\e{\frac{\alpha\overline x+\beta x}{\gamma}}=\chi_1(-\overline{a}b)\e{-\frac{\alpha a \overline {bf_2}+\beta a\overline b f_{2}}{f_1}}\chi_2(x)\e{\frac{\alpha\overline {f_1x}+\beta \overline {f_1}x}{f_2}}.
}
Thus, applying~\eqref{kseq2} in the case $c=1$ we obtain
\est{
\sum_{\substack{x\in I,\\(x,\gamma)=1,\\ (ax+b,c)=d}}\!\chi(x)\e{\frac{\alpha\overline x+\beta x}{\gamma}}&\ll \sum_{r|c}\Big(\gamma^\eps\frac{h_1'}{h'}\pr{\frac{\gamma_1'}{(\alpha,\gamma_1')}}^{\frac1{2}}+(\alpha,\gamma_1')^{\frac12}{\gamma_1'}^{\frac12+\eps}\frac{X}{\gamma_1' k'}\bigg)\\[-1.5em]
&\ll (\gamma c)^\eps\frac{h_1}{h}\pr{\frac{\gamma_1}{(\alpha,\gamma_1)}}^{\frac1{2}}+(\alpha,\gamma_1)^{\frac12}{\gamma_1}^{\frac12+\eps}\frac{X}{\gamma_1 k},
}
where we used the notation 
\est{
&k':=[g_1,k],\quad h'=({k'}, f_2):=(k,f_2),\quad h_1':=({k'}^\infty, f_2)=(k^{\infty},f_2), \quad \gamma_1':=f_{2}/h_1'
}
since $(c,k)=1$. Moreover $\gamma_1'|\gamma_1$ and ${\gamma_1'}^{\frac12}k'= [g_1,k] f_2^{\frac12}(k^{\infty},f_2)^{-\frac12} \geq \gamma_1^{\frac12}k$  since $f_1|g_1$ and thus $[g_1,k]\geq [f_1,k]=f_1k/(k,f_1)\geq f_1^{\frac12}k(k,f_1)^{-\frac12}$.

The rest of the argument is essentially identical as in the case of~\eqref{kseq}, with the only difference being in the last step, where we need to use Weil's bound for both summands.
\end{proof}

\addresses
\end{document}